\documentclass[12pt,reqno]{amsart}
\usepackage{graphicx}
\usepackage{amsfonts,amsmath,amsthm}
\usepackage{amssymb,epsfig}
\usepackage{amsmath,amssymb}
\usepackage{color}
\usepackage{fullpage}
\usepackage{pgfplots}
\usepackage{float}

\usepackage[title]{appendix}

\newtheorem{thm}{Theorem}[section]
\newtheorem{cor}[thm]{Corollary}
\newtheorem{lem}[thm]{Lemma}

\theoremstyle{definition}
\newtheorem{defn}[thm]{Definition}
\theoremstyle{remark}
\newtheorem{rem}{Remark}

\newtheorem{exa}[thm]{Example}
\numberwithin{equation}{section}

\usepackage[sc,osf]{mathpazo}   

\usepackage[euler-digits,small]{eulervm}
\AtBeginDocument{\renewcommand{\hbar}{\hslash}}

\newcommand{\R}{{\mathbb R}}

\usepackage{fullpage} 
\usepackage{microtype} 

\usepackage[unicode=true]{hyperref}
\hypersetup{colorlinks = true}
\hypersetup{
     colorlinks,
     linkcolor={black!10!red},
     linkbordercolor = {black!100!red},
     citecolor={blue}
}

\begin{document}
\title[Vanishing discount]
{Vanishing discount problem and the additive eigenvalues on changing domains}
\thanks{The author is supported in part by NSF grant DMS-1664424 and NSF CAREER grant DMS-1843320 to Hung Vinh Tran.}
\begin{abstract} 
    We study the asymptotic behavior, as $\lambda\rightarrow 0^+$, of the state-constraint Hamilton--Jacobi equation
\begin{equation}
\begin{cases}
   \phi(\lambda) u_\lambda(x) +  H(x,Du_\lambda(x)) \leq 0 \qquad\text{in}\,\;(1+r(\lambda))\Omega,\\
   \phi(\lambda) u_\lambda(x) +  H(x,Du_\lambda(x)) \geq 0 \qquad\text{on}\;(1+r(\lambda))\overline{\Omega}.
\end{cases} \tag{$S_\lambda$}
\end{equation}
and the corresponding additive eigenvalues, or ergodic constant 
\begin{equation}
\begin{cases}
   H(x,Dv(x)) \leq c(\lambda) \qquad\text{in}\,\;(1+r(\lambda))\Omega,\\
   H(x,Dv(x)) \geq c(\lambda) \qquad\text{on}\;(1+r(\lambda))\overline{\Omega}.
\end{cases} \tag{$E_\lambda$}
\end{equation}
Here, $\Omega$ is a bounded domain of $ \mathbb{R}^n$, $\phi(\lambda), r(\lambda):(0,\infty)\rightarrow \mathbb{R}$ are continuous functions such that $\phi$ is nonnegative and $\lim_{\lambda\rightarrow 0^+} \phi(\lambda) = \lim_{\lambda\rightarrow 0^+} r(\lambda) = 0$. We obtain both convergence and non-convergence results in the convex setting. Moreover, we provide a very first result on the asymptotic expansion of the additive eigenvalue $c(\lambda)$ as $\lambda\rightarrow 0^+$. The main tool we use is a duality representation of solution with viscosity Mather measures.
\end{abstract}

\author{Son N. T. Tu}
\address[S. N.T. Tu]
{
Department of Mathematics, 
University of Wisconsin Madison, 480 Lincoln  Drive, Madison, WI 53706, USA}
\email{thaison@math.wisc.edu}

\date{\today}
\maketitle

\tableofcontents
\clearpage

\section{Introduction}
Let $\phi(\lambda):(0,\infty)\rightarrow (0,\infty)$ be continuous nondecreasing and $r(\lambda):(0,\infty)\rightarrow \mathbb{R}$ be continuous such that $\lim_{\lambda\rightarrow 0^+} \phi(\lambda) = \lim_{\lambda\rightarrow 0^+} r(\lambda) = 0$. We study the asymptotic behavior, as the discount factor $\phi(\lambda)$ goes to $0$, of the viscosity solutions to the following state-constraint Hamilton--Jacobi equation
\begin{equation}\label{eq:S_lambda}
\begin{cases}
   \phi(\lambda) u_\lambda(x) +  H(x,Du_\lambda(x)) \leq 0 \qquad\text{in}\;\;(1+r(\lambda))\Omega,\\
   \phi(\lambda) u_\lambda(x) +  H(x,Du_\lambda(x)) \geq 0 \qquad\text{on}\;(1+r(\lambda))\overline{\Omega}.
\end{cases} \tag{$S_\lambda$}
\end{equation}
Here, $\Omega$ is a bounded domain of $\mathbb{R}^n$. For simplicity, we will write $\Omega_\lambda = (1+r(\lambda))\Omega$ for $\lambda > 0$. Roughly speaking, along some subsequence $\lambda_j\rightarrow 0^+$, we obtain the limiting equation as a state-constraint ergodic problem:
\begin{equation}\label{eq:S_0}
\begin{cases}
    H(x,Du(x)) \leq c(0) \qquad\text{in}\;\Omega,\\
    H(x,Du(x)) \geq c(0) \qquad\text{on}\;\overline{\Omega}.
\end{cases} \tag{$S_0$}
\end{equation}
Here $c(0)$ is the so-called critical value (additive eigenvalue) defined as
\begin{equation}\label{eq:c(0)}
    c(0) = \inf \Big\lbrace c\in \mathbb{R}: H(x,Du(x)) \leq c\;\;\text{in}\;\Omega\;\text{has a solution} \Big\rbrace.
\end{equation}
This quantity is finite and indeed the infimum in \eqref{eq:c(0)} can be replaced by minimum under our assumptions. We want to study the convergence of $u_\lambda$, solution to \eqref{eq:S_lambda}, under some normalization, to solution of \eqref{eq:S_0}. It turns out this problem is interesting and challenging as it concerns both the vanishing discount and the rate of changing domains at the same time.

The selection problem for the vanishing discount problems on fixed domains was studied extensively in the literature recently. The first-order equations on the torus was obtained in \cite{davini_convergence_2016}, 
and the second-order equations on the torus were studied in \cite{ishii2017,mitake_selection_2017}. 
The problems in bounded domains with boundary conditions were proved in \cite{al-aidarous_convergence_2016, Ishii2017a}. 
The problem in $\mathbb{R}^n$ under additional assumptions that lead to the compactness of the Aubry set was studied in \cite{ishii_vanishing_2020}. 
For the selection problems with state-constraint boundary conditions, so far, there is only \cite{Ishii2017a} 
that deals with a fixed domain, and there is not yet any result studying the situation of the changing domains. It turns out that the problem is much more subtle as we have to take into account the changing domain factor appropriately. Surprisingly, we can obtain both convergence results and non-convergence results in this setting.

This result is an extension of the selection principle in the setting of changing domains. Generally speaking, known results assert that in the convex setting the whole family of solutions of the discounted problems, which are uniquely solved if the ambient space is compact, converges to a distinguished solution of the ergodic limit equation 
\begin{equation}\label{eq:erg}
    H(x,Du(x)) = c(0).
\end{equation}
We emphasize that \eqref{eq:erg} has multiple solutions, therefore it is a non-trivial problem to characterize the limiting solution.

We show the convergence for some natural normalization of solutions to \eqref{eq:S_lambda} together with characterizing their limits, related characterizations are done in \cite{ishii_vanishing_2020} for the case the domain is $\mathbb{R}^n$ and in \cite{davini_convergence_2016, ishii2017,Hung2019} for the case the domain is torus $\mathbb{T}^n = \mathbb{R}^n/ \mathbb{Z}^n$. We also discuss other related results concerning the asymptotic behavior of the additive eigenvalue of $H$ in $\Omega_\lambda$ as $\lambda\rightarrow 0^+$.

\subsection{Assumptions} In this paper, by a domain, we mean an open, bounded, connected subset of $\mathbb{R}^n$. Without loss of generality, we will always assume $0\in \Omega$. To have well-posedness for \eqref{eq:S_lambda}, one needs to have a comparison principle. For simplicity, we will use the following structural assumptions on $\Omega$, which was introduced in \cite{Capuzzo-Dolcetta1990}. 
\begin{itemize}
    \item[(A1)] $\Omega$ a bounded star-shaped (with respect to the origin) open subset of $\mathbb{R}^n$ and there exists some $\kappa > 0$ such that
\begin{equation}\label{condA2}
    \mathrm{dist}(x,\overline{\Omega}) \geq \kappa r \qquad\text{for all}\; x\in (1+r) \partial\Omega, \;\text{for all}\;r>0. 
\end{equation}
\end{itemize}
It is worth noting that the first condition under which the comparison principle holds is the following, first introduced in \cite{Soner1986}.
\begin{itemize}
\item[(A2)] There exists a universal pair of positive numbers $(r,h)$ and $\eta\in \mathrm{BUC}(\overline{\Omega};\mathbb{R}^n)$ such that $B(x+t\eta(x), rt)\subset\Omega$ for all $x\in \overline{\Omega}$ and $t\in (0,h]$.
\end{itemize}
\begin{rem}\label{rem:Ishii} The assumption $\Omega$ is star-shaped can be removed in $\mathrm{(A1)}$, that is any bounded, open subset of $\mathbb{R}^n$ containing the origin that satisfies \eqref{condA2} for some $\kappa > 0$ is star-shaped, and  $\mathrm{(A2)}$ is a consequence of  $\mathrm{(A1)}$ (see Lemma \ref{thm:Ishii} in Appendix). Also $\mathrm{(A2)}$ can be generalized to a weaker \emph{interior cone} condition instead, that is there exists $\sigma\in  (0,1)$ such that $B(x+t\eta(x), rt^\sigma)\subset\Omega$ for all $x\in \overline{\Omega}$ and $t\in (0,h]$.
\end{rem}
We consider the following case in our paper about the vanishing and changing domain rates:
\begin{equation}\label{eq:asm}
    \lim_{\lambda\rightarrow 0^+}\left( \frac{r(\lambda)}{\phi(\lambda)}\right) = \gamma  \in [-\infty,+\infty].
\end{equation}
\begin{rem}\label{rem:first} Under the assumption \eqref{eq:asm}, there are only three possible cases:
\begin{enumerate}
    \item[1.] (Inner approximation) $r(\lambda)$ is negative for $\lambda\ll 1$, consequently $\gamma\leq 0$.
    \item[2.] (Outer approximation) $r(\lambda)$ is positive for $\lambda\ll 1$, consequently $\gamma\geq 0$.
    \item[3.] $r(\lambda)$ is oscillating around 0 when $\lambda\rightarrow 0^+$, consequently $\gamma = 0$. An example for this case is $r(\lambda) = \lambda\sin\left(\lambda^{-1}\right)$.
\end{enumerate}
We note that assumption \eqref{eq:asm} does not cover the case where $r(\lambda)/\phi(\lambda)$ is bounded but the limit at $\lambda\rightarrow 0^+$ does not exist, for example $r(\lambda) = \lambda\sin\left(\lambda^{-1}\right)$ and $\phi(\lambda) = \lambda$. Nevertheless, when the limit \eqref{eq:asm} exists and $r(\lambda)$ is oscillating near $0$, the limit must be $\gamma = 0$ and it turns out that the case $\gamma = 0$ is substantially simpler to analyze, as solutions of \eqref{eq:def_ulambda} converge to the maximal solution of \eqref{eq:S_0} (Theorem \ref{thm:subcritical}). 
\end{rem}

Throughout the paper, we will assume that $H:\overline{U}\times \mathbb{R}^n\rightarrow \mathbb{R}$ is a continuous Hamiltonian where $U = B(0,R_0)$ such that $2\Omega \subseteq U$. We list the main assumptions that will be used throughout the paper.
\begin{itemize}

\item[(H1)] For each $R>0$ there exists a constant $C_R$ such that
\begin{equation}\label{H3c}
\begin{cases}
|H(x,p) - H(y,p)| \leq C_R|x-y|,\\
|H(x,p) - H(x,q)| \leq C_R|p-q|,
\end{cases} \tag{H1}
\end{equation}
for $x,y \in \overline{U}$ and $p,q \in \mathbb{R}^n$ with $|p|,|q|\leq R$.

\item[(H2)] $H$ satisfies the coercivity assumption
\begin{equation}\label{H4}
\lim_{|p|\rightarrow \infty} \left(\min_{x\in \overline{U}} H(x,p)\right) = +\infty. \tag{H2}
\end{equation}
\item[(H3)]\label{H5} $p\mapsto H(x,p)$ is convex for each $x\in\overline{U}$.
\item[(H4)]\label{H6} For $v\in \mathbb{R}^n$, $x\mapsto L(x,v)$ is continuously differentiable on $\overline{U}$, where the Lagrangian $L$ of $H$ is defined as
\begin{equation*}
    L(x,v) = \sup_{p\in \mathbb{R}^n}\Big(p\cdot v - H(x,p)\Big), \qquad (x,v)\in \overline{U}\times \mathbb{R}^n.
\end{equation*}
\end{itemize}

The regularity assumption $\mathrm{(H4)}$ is needed for technical reason when we deal with changing domains, it satisfies for a vast class of Hamiltonians, for example $H(x,p) = H(p)+V(x)$ or $H(x,p) = V(x)H(p)$ with $V\in \mathrm{C}^1$.
\begin{rem} In fact we only need that $\lambda\mapsto L((1+\lambda)x,v)$ is continuously differentiable at $\lambda = 0$ but we assume $\mathrm{(H4)}$ for simplicity.
\end{rem}

\subsection{Literature on state-constraint and vanishing discount problems} There is a vast amount of works in the literature on the well-posedness of state-constraint Hamilton-Jacobi equations and fully nonlinear elliptic equations. The state-constraint problem for first-order convex Hamilton-Jacobi equations using optimal control frameworks was first studied in \cite{Soner1986, Soner1986a}. The general nonconvex, coercive first-order equations were then discussed in \cite{Capuzzo-Dolcetta1990}. For the nested domain setting, a rate of convergence for the discount problem is studied in \cite{kim2019stateconstraint}. We also refer to the classical books \cite{Bardi1997, Barles1994}, and the references therein.

There are also many works in the spirit of looking at a general framework of the vanishing discount problem. The convergence of solutions to the vanishing discount problems is proved in \cite{Ishii2017a}. A problem with a similar spirit to ours is considered in \cite{Qinbo2019}, in which the authors study the asymptotic behavior of solution on compact domain with respect to the Hamiltonian. In this work we take advantage of the clear from and structure of \eqref{eq:S_lambda} to obtain more explicit properties on solutions and furthermore the asymptotic expansion of the additive eigenvalues. We also remark that the continuity of the additive eigenvalue for general increasing domains for second-order equation is concerned in \cite{barles_large_2010}. See \cite{ishii_vanishing_2021} for a recent work on vanishing discount for weakly coupled system and \cite{murat_ergodic_2021} for second-order equation with Neumann boundary condition.

\subsection{Main results} There are two natural normalizations for solutions of \eqref{eq:S_lambda}. The first one is similar to what has been considered in \cite{Ishii2017a,ishii_vanishing_2020, Hung2019} as
\begin{equation}\label{eq:familyc0}
    \left\lbrace u_\lambda +\frac{c(0)}{\phi(\lambda)}\right\rbrace_{\lambda>0},
\end{equation}
and the second one is given by
\begin{equation}\label{eq:familyclambda}
    \left\lbrace u_\lambda+\frac{c(\lambda)}{\phi(\lambda)}\right\rbrace_{\lambda>0},
\end{equation}
where $c(\lambda)$ is the additive eigenvalues of $H$ in $\Omega_\lambda$, as defined in equation \eqref{eq:cell3}. Let $u^0$ be the limiting solution of the vanishing discount problem on fixed bounded domain (see \cite{davini_convergence_2016, ishii2017,Ishii2017a,Hung2019} and Theorem \ref{thm:conv_bdd}), our first result is as follows.

\begin{thm}\label{thm:subcritical} Assume \eqref{H3c}, \eqref{H4}, $\mathrm{(H3)}$ and $\mathrm{(A1)}$. If $\gamma = 0$ then both families \eqref{eq:familyc0} and \eqref{eq:familyclambda} converge to $u^0$ locally uniformly as $\lambda \rightarrow 0^+$.
\end{thm}
We note that Theorem \ref{thm:subcritical} includes the case where $r(\lambda)$ is oscillating, as long as the limit \eqref{eq:asm} exists. For example $r(\lambda) = \lambda \sin\left(\lambda^{-1}\right)$ and $\phi(\lambda) = \lambda^p$ with $p\in (0,1)$.

If $\gamma$ is finite then \eqref{eq:familyc0} is bounded and convergent. Its limit can be characterized in terms of probability minimizing measures $\mathcal{M}_0$ (or viscosity Mather measures, see Section 2). For a ball $\overline{B}_h\subset \mathbb{R}^n$ and a measure $\mu$ defined on $\overline{\Omega}\times \overline{B}_h$, we define 
\begin{equation}\label{def:integral}
    \langle \mu, f\rangle := \int_{\overline{\Omega}\times \overline{B}_h} f(x,v)\;d\mu(x,v), \qquad\text{for}\;f\in \mathrm{C}(\overline{\Omega}\times \overline{B}_h).
\end{equation}
\begin{thm}\label{thm:general} Assume \eqref{H3c}, \eqref{H4}, $\mathrm{(H3)},  \mathrm{(H4)}$ and $\mathrm{(A1)}$. If $\gamma \in \mathbb{R}$ then the family \eqref{eq:familyc0} converge to $u^\gamma$ locally uniformly in $\Omega$ as $\lambda\rightarrow 0^+$. Furthermore
\begin{equation}\label{eq:thm}
    u^\gamma = \sup_{w\in \mathcal{E}^\gamma} w,
\end{equation}
where $\mathcal{E}^\gamma$ denotes the family of subsolutions $w$ to the ergodic problem \eqref{eq:S_0} such that
\begin{equation*}
  \gamma\big\langle \mu,  (-x)\cdot D_xL(x,v)\big\rangle +\langle \mu, w\rangle \leq 0
  \qquad\text{for all}\; \mu\in \mathcal{M}_0.
\end{equation*}
\end{thm}

\begin{rem} The factor $\gamma\left\langle \mu, (-x)\cdot D_xL(x,v)\right\rangle$ here captures the scaling property of the problem, which is where $u^\gamma$ and $u^0$ are different from each other. Also, if $\gamma = \infty$ then the family \eqref{eq:familyc0} could be unbounded (Example \ref{ex:a}). We note that Theorem \ref{thm:general} includes the conclusion of Theorem \ref{thm:subcritical} for the family \eqref{eq:familyc0} but we do not need the technical assumption $\mathrm{(H4)}$ for Theorem \ref{thm:subcritical}. 
\end{rem}

\begin{cor}\label{cor:mycor} The mapping $\gamma\mapsto u^\gamma(\cdot)$ from $\mathbb{R}$ to $\mathrm{C}(\overline{\Omega})$ is concave and decreasing. Precisely, if $\alpha,\beta\in \mathbb{R}$ with $\alpha\leq \beta$ then $u^\beta\leq u^\alpha$ and
\begin{equation*}
    (1-\lambda)u^\alpha + \lambda u^\beta \leq u^{(1-\lambda)\alpha+\lambda\beta} \qquad\text{for every}\;\lambda\in [0,1].
\end{equation*}
\end{cor}

For the second family \eqref{eq:familyclambda}, we observe that it is bounded even if $\gamma = \infty$, and the difference between the two normalization \eqref{eq:familyclambda} and \eqref{eq:familyc0} is given by
\begin{equation}\label{eq:introlimit}
    \left\lbrace \frac{c(\lambda)-c(0)}{\phi(\lambda)}\right\rbrace_{\lambda>0}. 
\end{equation}
If $\gamma<\infty$ then the two families \eqref{eq:familyc0} and \eqref{eq:familyclambda} are convergent if and only if the limit of \eqref{eq:introlimit} as $\lambda\rightarrow 0^+$ exists. In that case we have 
\begin{equation}\label{eq:introlimit2}
   \lim_{\lambda\rightarrow 0^+}\left(\frac{c(\lambda)-c(0)}{\phi(\lambda)}\right)= \gamma\lim_{\lambda\rightarrow 0^+}\left(\frac{c(\lambda)-c(0)}{r(\lambda)}\right).
\end{equation}
The limit on the right-hand side should be understood as taking along sequences where $r(\lambda)\neq 0$. In other words we only concern those functions $r(\cdot)$ that are not identically zero near $0$, since otherwise $c(\lambda) = c(0)$ for $\lambda\ll 1$ and the problem is not interesting. It leads naturally to the question of the asymptotic expansion of the critical value
\begin{equation}\label{asymptotic}
c(\lambda) = c(0) + c^{(1)} r(\lambda) + o(r(\lambda)) \qquad\text{as}\; \lambda \rightarrow 0^+.
\end{equation}
To our knowledge, this kind of question is new in the literature. We prove that the limit in \eqref{eq:introlimit2} always exists if $r(\lambda)$ does not oscillate its sign near $0$, and as a consequence it provides a necessary and sufficient condition under which the limit \eqref{eq:introlimit2} exists for a general oscillating $r(\lambda)$. Of course this oscillating behavior is excluded when we only concern about the convergence of \eqref{eq:familyc0} and \eqref{eq:familyclambda} (since $\gamma = 0$). We also give a characterization for the limit in \eqref{eq:introlimit2} in terms of $\mathcal{M}_0$.

\begin{thm}\label{thm:limit} Assume \eqref{H3c}, \eqref{H4}, $\mathrm{(H3)}, \mathrm{(H4)}$, $\mathrm{(A1)}$, we have
\begin{align*}
    &\lim_{\substack{\lambda\rightarrow 0^+\\ r(\lambda) > 0}} \left(\frac{c(\lambda) - c(0)}{r(\lambda)}\right) = \max_{\mu\in \mathcal{M}_0} \left\langle \mu, (-x)\cdot D_xL(x,v)\right\rangle,\\
    &\lim_{\substack{\lambda\rightarrow 0^+\\ r(\lambda) < 0}} \left(\frac{c(\lambda) - c(0)}{r(\lambda)}\right) = \min_{\mu\in \mathcal{M}_0} \left\langle \mu, (-x)\cdot D_xL(x,v)\right\rangle.
\end{align*}
Thus $(c(\lambda)-c(0))/r(\lambda)$ converges as $\lambda\rightarrow 0^+$ if and only if the following invariant holds
\begin{equation*}
    \left\langle \mu,(-x)\cdot D_xL(x,v)\right\rangle = c^{(1)} \qquad\text{for all}\; \mu\in \mathcal{M}_0
\end{equation*}
where $c^{(1)}$ is a positive constant.
\end{thm}


\begin{cor}\label{cor:Aug30-1} If $\left\langle \mu,(-x)\cdot D_xL(x,v)\right\rangle = c^{(1)}$ for all $\mu\in \mathcal{M}_0$ then $u^\gamma(\cdot) +\gamma c_{(1)} = u^0(\cdot)$.
\end{cor}

\begin{cor}\label{cor:equala} If $u^0(z) = u^\gamma(z)$ for some $z\in \Omega$ and $\gamma > 0$ then $ c^{(1)}_-= 0$.
\end{cor}

Theorem \ref{thm:limit} gives us the convergence of the second normalization \eqref{eq:familyclambda} for finite $\gamma$. We recall that the case $\gamma = 0$ is already considered in Theorems \ref{thm:subcritical} and \ref{thm:general}.

\begin{cor}\label{cor:second_norm} Assume \eqref{H3c}, \eqref{H4}, $\mathrm{(H3)}, \mathrm{(H4)}$, $\mathrm{(A1)}$ and $\gamma \in \mathbb{R}\backslash \{0\}$, then 
\begin{equation*}
    \lim_{\lambda\rightarrow 0^+} \left( u_\lambda(x) + \frac{c(\lambda)}{\phi(\lambda)} \right) =  u^\gamma(x) + \gamma\lim_{\lambda\rightarrow 0^+} \left(\frac{c(\lambda)-c(0)}{r(\lambda)}\right)
\end{equation*}
locally uniformly in $\Omega$.
\end{cor}

Even though the second normalization \eqref{eq:familyclambda} remains uniformly bounded when $\gamma = \pm\infty$, it is rather surprising that we have a divergent result in this convex setting. Using tools from weak KAM theory, we can construct an example where divergence happens when approximating from the inside. To our knowledge, this kind of example is new in the literature.

 \begin{thm}\label{thm:counter-example} There exists a Hamiltonian where given any $r(\lambda)\leq 0$ we can construct $\phi(\lambda)$ such that along a subsequence $\lambda_j \rightarrow 0^+$ we have $r(\lambda_j)/\phi(\lambda_j) \rightarrow -\infty$ and \eqref{eq:familyclambda} diverges.
\end{thm}

\subsection{Organization of the paper} 
The paper is organized in the following way. In Section \ref{sec2}, we provide the background results on state-constraint Hamilton--Jacobi equations, the duality representation, and some weak KAM theory backgrounds that will be needed throughout the paper. We also review the selection principle for vanishing discount on a fixed bounded domain and its characterization of the limit, as well as set up the problems of vanishing discount on changing domains. Section \ref{sec3} is devoted to proving the main results (Theorems \ref{thm:general}) on the convergence of the first normalization, together with some examples. In Section \ref{sec4} we provide the proof for Theorem \ref{thm:limit} which addresses the question of asymptotic expansion of the eigenvalue, which relates the limits of different normalization in vanishing discount and proof for Corollaries \ref{cor:Aug30-1} and \ref{cor:equala}. In Section \ref{sec5} we restate the convergence of the second normalization as a consequence from Sections \ref{sec3} and \ref{sec4}, and provide a new counterexample where divergence happens (Theorem \ref{thm:counter-example}). The proofs of some Theorems and Lemmas are provided in the Appendix.

\section{Preliminaries}\label{sec2}

\subsection{State-constraint solutions}
For an open subset $\Omega \subset U \subset\mathbb{R}^n$, we denote the space of  bounded uniformly continuous functions defined in $\Omega$ by $\mathrm{BUC}(\Omega;\mathbb{R})$. Assume that $H:\overline{U}\times \mathbb{R}^n \rightarrow \mathbb{R}$ is a Hamiltonian such that $H\in \mathrm{BUC}(\overline{U}\times \mathbb{R}^n)$ satisfying $\mathrm{(H1)}, \mathrm{(H2)}$. We consider the following equation with $\delta \geq 0$:
\begin{equation}\label{HJ-static}
    \delta u(x) + H(x,Du(x)) = 0 \qquad\text{in}\;\Omega \tag{HJ}.
\end{equation}
\begin{defn}\label{defn:1} We say that
\begin{itemize}
\item[(i)] $v\in \mathrm{BUC}(\Omega;\mathbb{R})$ is a viscosity subsolution of \eqref{HJ-static} in $\Omega$ if, for every $x\in \Omega$ and $\varphi\in \mathrm{C}^1(\Omega)$ such that $v-\varphi$ has a local maximum over $\Omega$ at $x$, $\delta v(x) + H\big(x,D\varphi(x)\big) \leq 0$ holds.

\item[(ii)] $v\in \mathrm{BUC}(\overline\Omega;\mathbb{R})$ is a viscosity supersolution of \eqref{HJ-static} on $\overline\Omega$ if, for every $x\in \overline{\Omega}$ and $\varphi\in \mathrm{C}^1(\overline{\Omega})$ such that $v-\varphi$ has a local minimum over $\overline{\Omega}$ at $x$, $\delta v(x) + H\big(x,D\varphi(x)\big) \geq 0$ holds.
\end{itemize}

If $v$ is a viscosity subsolution to \eqref{HJ-static} in $\Omega$, and is a viscosity supersolution to \eqref{HJ-static} on $\overline{\Omega}$, that is, $v$ is a viscosity solution to
\begin{equation}\label{state-def}
\begin{cases}
\delta v(x) + H(x,Dv(x)) \leq 0 &\quad\text{in}\; \Omega,\\
\delta v(x) + H(x,Dv(x)) \geq 0 &\quad\text{on}\; \overline{\Omega},
\end{cases} \tag{HJ$_\delta$}
\end{equation}
then we say that $v$ is a state-constraint viscosity solution of \eqref{HJ-static}. 
\end{defn}

\begin{defn} For a real valued function $w(x)$ define for $x\in \Omega$, we define the super-differential and sub-differential of $w$ at $x$ as
\begin{align*}
D^{+}w(x)&=\left \lbrace p \in \R^n : \limsup_{y \rightarrow x} \frac{w(y)-w(x)-p\cdot (y-x)}{|y-x|} \leq 0\right\rbrace,\\
D^{-}w(x)&=\left \lbrace p \in \R^n : \liminf_{y \rightarrow x} \frac{w(y)-w(x)- p \cdot (y-x)} {|y-x|} \geq 0 \right\rbrace.
\end{align*}
\end{defn}

We refer the readers to \cite{ Capuzzo-Dolcetta1990,kim2019stateconstraint,Soner1986} for the existence and wellposedness, as well as the Lipschitz bound on solutions $u_\lambda$ of \eqref{eq:S_lambda} and a description on state-constraint boundary condition. See also \cite{Bardi1997,Barles1994, Hung2019} for the equivalent definition of viscosity solution using super-differential and sub-differential.

\begin{thm}\label{Perron} Assume \eqref{H4} and $\delta > 0$. Then, there exists a state-constrained viscosity solution $u\in \mathrm{C}(\overline{\Omega})\cap \mathrm{W}^{1,\infty}(\Omega)$ to \eqref{state-def} with $\delta|u(x)|+ |Du(x)| \leq C_H$ for $x\in \Omega$ where $C_H$ only depends on $H$.
\end{thm}

\begin{cor}\label{cor:Inv_max} Let $u\in \mathrm{C}(\overline{\Omega})$ be a viscosity subsolution to \eqref{HJ-static} in $\Omega$ with $\delta>0$. If $v\leq u$ on $\overline{\Omega}$ for all viscosity subsolutions $v\in \mathrm{C}(\overline{\Omega})$ of \eqref{HJ-static} in $\Omega$ then $u$ is a viscosity supersolution to \eqref{HJ-static} on $\overline{\Omega}$.
\end{cor}


\begin{thm}\label{CP continuous} 
Assume \eqref{H3c}, $\mathrm{(A2)}$ and $\delta > 0$. If $v_1\in \mathrm{BUC}(\overline{\Omega};\mathbb{R})$ is a Lipschitz viscosity subsolution of \eqref{HJ-static} in $\Omega$, and $v_2\in \mathrm{BUC}(\overline{\Omega};\mathbb{R})$ is a viscosity supersolution of \eqref{HJ-static} on $\overline{\Omega}$ then $v_1(x)\leq v_2(x)$ for all $x\in \overline{\Omega}$.
\end{thm}
When the uniqueness of \eqref{state-def} is guaranteed, the unique viscosity solution to \eqref{state-def} is the maximal viscosity subsolution of \eqref{HJ-static}.

\subsection{Duality representation of solutions} 

The duality representation is well-known in the literature (see \cite{Ishii2017a} or \cite[Theorem 5.3]{Hung2019}). We present here a variation of that result. For $\delta > 0$, let $u_\delta$ be the unique solution to \eqref{state-def}, we have the following bound:
\begin{equation}\label{eq:priori}
    \delta|u_\delta(x)| + |Du_\delta(x)| \leq C_H\qquad \text{for all}\; x\in \Omega.
\end{equation}
That means the value of $H(x,p)$ for large $|p|$ is irrelevant, therefore without loss of generality we can assume that there exists $h>0$ such that, the Legendre's transform $L$ of $H$ will satisfy: 
\begin{equation}
\begin{cases}
\displaystyle H(x,p) = \sup_{|v|\leq h} \Big(p\cdot v - L(x,v)\Big), &\qquad (x,p)\in \overline{\Omega}\times\mathbb{R}^n\\
\displaystyle L(x,v) = \sup_{p\in \mathbb{R}^n} \Big(p\cdot v - H(x,p)\Big), &\qquad (x,v)\in \overline{\Omega}\times \overline{B}_h.
\end{cases} \label{eq:H&L}
\end{equation}
This simplification allows us to work with the compact subset $\overline{\Omega}\times \overline{B}_h$ rather than $\overline{\Omega}\times \mathbb{R}^n$, as will be utilized to obtain the duality representation. Let us define for each $f\in \mathrm{C}(\overline{\Omega}\times \overline{B}_h)$ the function
\begin{equation*}
H_f(x,p) = \max_{|v|\leq h} \Big(p\cdot v - f(x,v)\Big), \qquad (x,p)\in \overline{\Omega}\times \overline{B}_h.
\end{equation*}
Recall the definition of the action $\langle \cdot, \cdot\rangle$. The underlying domain of the integral will be implicitly understood. Let $\mathcal{R}(\overline{\Omega}\times \overline{B}_h)$ be the space of Radon measures on $\overline{\Omega}\times \overline{B}_h$. For $\delta>0$, $z\in \overline{\Omega}$ we define
\begin{align*}
\mathcal{F}_{\delta,\Omega} &= \Big\lbrace (f,u) \in \mathrm{C}(\overline{\Omega}\times \overline{B}_h)\times \mathrm{C}(\overline{\Omega}): \delta u + H_f(x,Du)\leq 0\;\text{in}\;\Omega \Big\rbrace\\
 \mathcal{G}_{z,\delta,\Omega} &= \Big\lbrace f - \delta u(z): (f, u)\in \mathcal{F}_{\delta,\Omega} \Big\rbrace\\
 \mathcal{G}'_{z,\delta,\Omega} &= \Big\lbrace \mu\in \mathcal{R}(\overline{\Omega}\times \overline{B}_h): \langle \mu, f\rangle \geq 0\;\text{for all}\;f\in \mathcal{G}_{z,\delta,\Omega} \Big\rbrace.
\end{align*}
Here $\mathcal{G}_{z,\delta,\Omega} \subset \mathrm{C}(\overline{\Omega}\times \overline{B}_h)$ is the evaluation cone of $\mathcal{F}_{\delta,\Omega}$, and its dual cone consists of Radon measures with non-negative actions against elements in $\mathcal{G}_{z,\delta, \Omega}$. Note that $\mathcal{R}(\overline{\Omega}\times \overline{B}_h)$ is the dual space of $\mathrm{C}(\overline{\Omega}\times \overline{B}_h)$. We also denote $\mathcal{P}$ the set of probability measures.
\begin{lem}\label{lem:cv} $\mathcal{F}_{\delta,\Omega}$ is a convex set, $\mathcal{G}_{z,\delta,\Omega}$ is a convex cone with vertex at the origin, and $\mathcal{G}_{z,\delta,\Omega}'$ consists of only non-negative measures.
\end{lem}
\begin{thm}\label{thm:lambdau} For $(z,\delta) \in \overline{\Omega}\times (0,\infty)$ and $u$ is the viscosity solution to \eqref{state-def}, we have
\begin{equation*}
\delta u(z) = \min_{\mu\in \mathcal{P} \cap \mathcal{G}'_{z,\delta,\Omega}} \langle \mu, L\rangle = \min_{\mu\in \mathcal{P} \cap \mathcal{G}'_{z,\delta,\Omega}} \int_{\overline{\Omega}\times\overline{B}_h} L(x,v)\;d\mu(x,v).
\end{equation*}
\end{thm}
As $\delta \rightarrow 0^+$, we also have a representation for the erogdic problem \eqref{eq:S_0} in the same manner. Let us define
\begin{align*}
\mathcal{F}_{0,\Omega} &= \Big\lbrace (f, u)\in \mathrm{C}(\overline{\Omega}\times \overline{B}_h)\times \mathrm{C}(\overline{\Omega}): H_f(x,Du(x)) \leq 0\;\text{in}\;\Omega \Big\rbrace\\
\mathcal{G}_{0,\Omega} &= \Big\lbrace f: (f, u)\in \mathcal{F}_{0,\Omega}\;\text{for some}\;u\in \mathrm{C}(\overline{\Omega}) \Big\rbrace\\
\mathcal{G}'_{0,\Omega} &= \Big\lbrace \mu\in \mathcal{R}(\overline{\Omega}\times \overline{B}_h): \langle \mu, f\rangle \geq 0\;\text{for all}\;f\in \mathcal{G}_{0,\Omega} \Big\rbrace.
\end{align*}
Here the notion of viscosity subsolution is equivalent to a.e. subsolution in $\Omega$, thanks to $\mathrm{(H4)}$. We also have $\mathcal{G}_{0,\Omega} \subset\mathrm{C}(\overline{\Omega}\times \overline{B}_h)$ is the evaluation cone of $\mathcal{F}_{0,\Omega}$ and $\mathcal{G}'_{0,\Omega}$ is the dual cone of $\mathcal{G}_{0,\Omega}$ in $\mathcal{R}(\overline{\Omega}\times \overline{B}_h)$.

\begin{defn} A measure $\mu$ defined on $\overline{\Omega}\times \overline{B}_h$ is called a \emph{holonomic measures} if
\begin{equation*}
\langle \mu, v\cdot D\psi(x) \rangle = 0 \qquad\text{for all}\; \psi\in \mathrm{C}^1(\overline{\Omega}).
\end{equation*}
\end{defn}

\begin{lem} Measures in $\mathcal{G}'_{0,\Omega}$ are holonomic.
\end{lem}
\begin{proof} If $\psi\in \mathrm{C}^1(\Omega)$ then $\pm(v\cdot D\psi(x), \psi)\in \mathcal{F}_{0,\Omega}$, therefore $\pm v\cdot D\psi(x) \in \mathcal{G}_{0,\Omega}$ and thus $\langle \mu, v\cdot D\psi(x) \rangle = 0$.
\end{proof}

\begin{lem}\label{lem:cv3} Fix $z\in \overline{\Omega}$ and $\delta_j\rightarrow 0$. Assume $u_j\in \mathcal{G}_{z,\delta_j,\Omega}'$ and $\mu_j \rightharpoonup \mu$ weakly in the sense of measures, then $\mu\in \mathcal{G}_{0,\Omega}'$.
\end{lem}

\begin{lem}\label{lem:cv2} $\mathcal{F}_{0,\Omega}$ is a convex set, $\mathcal{G}_{0,\Omega}$ is a convex cone with vertex at the origin, and $\mathcal{G}_{0,\Omega}'$ consists of only nonnegative measures.
\end{lem}
\begin{thm}\label{thm:c0} We have
\begin{equation}\label{eq:Mather}
-c(0) = \min_{\mu\in \mathcal{P}\cap \mathcal{G}_{0,\Omega}'} \langle \mu, L \rangle =\min_{\mu\in \mathcal{P}\cap \mathcal{G}_{0,\Omega}'} \int_{\overline{\Omega}\times \overline{B}_h} L(x,v)\;d\mu(x,v) .
\end{equation}
\end{thm}
The set of all measures in $\mathcal{P}\cap \mathcal{G}'_{0,\Omega}$ that minimizing \eqref{eq:Mather} is denoted $\mathcal{M}_0$. We call them viscosity Mather measures (\cite{Ishii2017a}). We omit the proofs of Lemmas \ref{lem:cv}, \ref{lem:cv3}, \ref{lem:cv2} and Theorems \ref{thm:lambdau}, \ref{thm:c0} as they are slight modifications of those in the periodic setting, which we refer the interested readers to \cite{Ishii2017a,Hung2019}.

\subsection{Vanishing discount for fixed bounded domains}

In this section we use the representation formulas in Theorem \ref{thm:lambdau} and Theorem \ref{thm:c0} to show the convergence of solution of \eqref{eq:S_lambda} to solution of \eqref{eq:S_0}. See also \cite{Ishii2017a,Hung2019} where the similar technique is used.

\begin{thm}\label{thm:pre} Assume \eqref{H3c}, \eqref{H4} and $\mathrm{(A2)}$. Let $u_\delta\in \mathrm{C}(\overline{\Omega})\cap\mathrm{Lip}(\Omega)$ be the unique solution to \eqref{state-def}. Then $\delta u_\delta(\cdot) \rightarrow -c(0)$ uniformly on $\overline{\Omega}$ as $\delta\rightarrow 0$. Indeed, there exists $C>0$ depends on $H$ and $\mathrm{diam}(\Omega)$ such that for all $x\in \overline{\Omega}$ there holds
\begin{equation}\label{eq:rate0}
\left|\delta u_\delta(x)+c(0)\right| \leq C\delta.
\end{equation}
Also, for each $x_0\in \overline{\Omega}$ there exist a subsequence $\lambda_j$ and $u\in \mathrm{C}(\overline{\Omega})$ solving \eqref{eq:S_0} such that:
\begin{equation*}
\begin{cases}
u_{\delta_j}(x) - u_{\delta_j}(x_0) \rightarrow u(x)\vspace{0.1cm}\\
 u_{\delta _j}(x) + c(0)/\delta_j \rightarrow w(x)
\end{cases} 
\end{equation*}
uniformly on $\overline{\Omega}$ as $\delta_j\rightarrow 0$ and the difference between the two limits are $w(x) - u(x) = w(x_0)$.
\end{thm}

\begin{thm}\label{thm:conv_bdd} Assume \eqref{H3c}, \eqref{H4}, $\mathrm{(H3)}$ and $\mathrm{(A2)}$. Let $u_\delta\in \mathrm{C}(\overline{\Omega})\cap\mathrm{Lip}(\Omega)$ be the unique solution to \eqref{state-def}. Then, $u_\delta+\delta^{-1}c(0) \rightarrow u^0$ uniformly on $\mathrm{C}(\overline{\Omega})$ as $\delta\rightarrow 0^+$ and $u^0$ solves \eqref{eq:S_0}. Furthermore, the limiting solution can be characterized as 
\begin{equation}\label{eq:char}
    u^0 = \sup_{v\in \mathcal{E}} v
\end{equation}
where $\mathcal{E}$ is the set of all subsolutions $v\in \mathrm{C}(\overline{\Omega})$ to $H(x,Dv(x))\leq c(0)$ in $\Omega$ such that $\langle \mu, v\rangle \leq 0$ for all $\mu\in \mathcal{M}_0$, the set of all minimizing measures $\mu\in \mathcal{P}\cap \mathcal{G}'_0$ such that $-c(0) = \langle \mu, L\rangle$.
\end{thm}

We provide the proof of Theorem \ref{thm:pre} in the Appendix. The proof of Theorem \ref{thm:conv_bdd} is omitted as it is a slight modification of the one in \cite{Hung2019}. The characterization \eqref{eq:char} also appears in \cite{davini_convergence_2016, ishii2017, ishii_vanishing_2020} under different settings.


\subsection{Maximal subsolutions and the Aubry set} For any domain (with nice boundary) $\Omega\subset U$, we recall that the additive eigenvalue of $H$ in $\Omega$ is defined as
\begin{equation*}
        c_\Omega = \inf \Big\lbrace c \in \mathbb{R}: H(x,Dv(x)) \leq c\;\text{has a viscosity subsolution in}\;\Omega \Big\rbrace.
\end{equation*} 
We consider the following equation
\begin{equation}\label{eqn:S_mu}
    H(x,Dv(x)) \leq c_\Omega \qquad 
    \text{in}\;\Omega\tag{$S_{\Omega}$}.
\end{equation}
We note that viscosity subsolutions of \eqref{eqn:S_mu} in $U$ are Lipschitz, and therefore they are equivalent to a.e. subsolutions (see \cite{barron_semicontinuous_1990,Hung2019}). Also it is clear that $c_\Omega\leq c_U$, where $c_U$ is the additive eigenvalue of $H$ in $U$. 

\begin{defn}\label{defn:m_mu} For a fixed $z\in \Omega$ as a vertex, we define
\begin{align*}
    S_{\Omega}(x,z) = \sup \big\lbrace v(x)-v(z): v\;\text{solves}\;\eqref{eqn:S_mu} \big\rbrace, \qquad x\in \Omega.
\end{align*}
There is a unique (continuous) extension $S_{\Omega}: \overline{\Omega}\times \overline{\Omega}\rightarrow \mathbb{R}$, we call $x\mapsto S_{\Omega}(x,z)$ the \emph{maximal subsolution} to \eqref{eqn:S_mu} with vertex $z$.
\end{defn}

\begin{thm}\label{thm:basic} \quad 
\begin{itemize}
    \item[(i)] For each fixed $z\in \overline{\Omega}$ then $x\mapsto S_\Omega(x,z)$ solves
    \begin{equation}\label{eq:sta}
    \begin{cases}
        H(x,Du(x)) \leq c_\Omega &\quad\;\text{in}\;\Omega,\\
        H(x,Du(x)) \geq c_\Omega &\quad\;\text{on}\;\overline{\Omega}\backslash\{z\}.
    \end{cases}
    \end{equation}
    \item[(ii)] We have the triangle inequality $ S_{\Omega}(x,z) \leq S_{\Omega}(x,y) + S_{\Omega}(y,z)$ for all $x,y,z \in \overline{\Omega}$.
\end{itemize}
We call $S_{\Omega}:\overline{\Omega}\times \overline{\Omega}\rightarrow \mathbb{R}$ an intrinsic semi-distance on $\Omega$ (see also \cite{Bardi1997,Barles1994,Ishii2008}). 
\end{thm}
In the sense of Definition \ref{defn:1}, inequality \eqref{eq:sta} means $x\mapsto u(x)$ is a subsolution to $H(x,Du(x))= c_\Omega$ in $\Omega$, and $x\mapsto u(x)$ is a supersolution to $H(x,Du(x))= c_\Omega$ in $\overline{\Omega}\backslash \{z\}$. We omit the proof of Theorems \ref{thm:basic} as it is a simple variation of Perron's method.

\begin{defn} Let us define the ergodic problem in $\Omega$ as
\begin{equation}\label{eq:E}
    \begin{cases}
    H(x,Du(x)) \leq c_\Omega &\quad\text{in}\;\Omega,\\ 
    H(x,Du(x)) \geq c_\Omega &\quad\text{on}\;\overline{\Omega}.
    \end{cases}\tag{E}
\end{equation}
The Aubry set $\mathcal{A}$ in $\Omega$ is defined as
\begin{equation*}
    \mathcal{A}_{\Omega} = \Big\lbrace z\in \overline{\Omega}: x\mapsto S_{\Omega}(x,z)\;\text{is a solution to}\;\eqref{eq:E}\Big\rbrace.
\end{equation*}
\end{defn}

\begin{thm}\label{thm:V} Assuming $H(x,p) = |p| - V(x)$ where $V\in \mathrm{C}(\overline{\Omega})$ is nonnegative. 
\begin{itemize}
    \item[(i)] The additive eigenvalue of $H$ in $\Omega$ is $c_\Omega  = -\min_{\overline{\Omega}}V$.
    \item[(ii)] The Aubry set of $H$ in  $\Omega$ is $\mathcal{A}_\Omega = \left\lbrace z\in \overline{\Omega}: V(z) = -c_\Omega = \min_{\overline{\Omega} }V \right\rbrace $.
\end{itemize}
\end{thm}

\begin{thm}\label{thm:char_Aubry} Given $z\in \Omega$, then $z\notin \mathcal{A}_\Omega$ if and only if there is a subsolution of $H(x,Du(x))\leq c_\Omega$ in $\Omega$ which is strict in some neighborhood of $z$. Here we say $u\in \mathrm{C}(\overline{\Omega})$ is a \emph{strict subsolution} to $H(x,Du(x)) = c_\Omega$ in $B(x_0,r)\subset \Omega$ if there exists some $\varepsilon>0$ such that $H(x,p) \leq c_\Omega - \varepsilon$ for all $p\in D^+u(x), x\in B(x_0,r)$.
\end{thm}

\begin{thm}\label{thm:eigenvalue} If $\mathcal{A}_U \subset\subset \Omega \subset U$ then the additive eigenvalue of $H$ in $\Omega$ is $c_\Omega = c_U$.
\end{thm}

We give proofs for \ref{thm:V} and \ref{thm:eigenvalue} in Appendix. A proof of Theorem \ref{thm:V} for the case $\Omega = \mathbb{R}^n$ can be found in \cite{Hung2019}. Theorem \ref{thm:eigenvalue} is taken from \cite[Proposition 5.1]{ishii_vanishing_2020}. Proof of Theorem \ref{thm:char_Aubry} can be found in \cite{ishii_vanishing_2020}.

The maximal solution $S_\Omega(x,y)$ also has an optimal control formulation (minimal exists time) as follows (see \cite{fathi_pde_2005, ishii_vanishing_2020}).

\begin{thm}[Optimal control formula]\label{lem:optimal} Let us define for $\beta \geq \alpha \geq 0$ the following set:
\begin{equation*}
    \mathcal{F}_{\Omega}(x,y; \alpha,\beta) = \Big\lbrace\xi \in \mathrm{AC}\left([0,T], \overline{\Omega}\right);  T > 0, \xi(\alpha) = y, \xi(\beta) = x \Big\rbrace.
\end{equation*}
Then
\begin{equation*}
S_{\Omega}(x,y) = \inf \left\lbrace \int_0^T \Big( c(0)+  L(\xi(s),\dot{\xi}(s))\Big)\;ds: \xi\in \mathcal{F}_{\Omega}(x,y;0, T)   \right\rbrace.
\end{equation*}
\end{thm}

\subsection{The vanishing discount problem on changing domains} Let $\Omega_\lambda = (1+r(\lambda))\Omega$. For each $\lambda \in (0,1)$ let $u_\lambda \in \mathrm{BUC}(\overline{\Omega}_\lambda)\cap \mathrm{Lip}(\Omega_\lambda)$ be the unique viscosity state-constraint solutions to 
\begin{equation}\label{eq:def_ulambda}
\begin{cases}
\phi(\lambda) u_\lambda(x) + H(x,Du_\lambda(x)) \leq 0 &\quad\text{in}\;\Omega_\lambda,\\
\phi(\lambda) u_\lambda(x) + H(x,Du_\lambda(x)) \geq 0 &\quad\text{on}\;\overline{\Omega}_\lambda.
\end{cases}
\end{equation}
The additive eigenvalue $c(\lambda)$ of $H$ in $\Omega_\lambda$ is the unique constant such that the following ergodic problem can be solved
\begin{equation}\label{eq:cell3}
\begin{cases}
H(x,Du(x)) \leq c(\lambda) &\quad \text{in}\;\Omega_\lambda\\
H(x,Du(x)) \geq c(\lambda) &\quad \text{on}\;\overline{\Omega}_\lambda.
\end{cases}
\end{equation}
By comparison principle, it is clear that if $r(\lambda)\geq 0$ then $c(0)\leq c(\lambda)$ and if $\lambda\mapsto r(\lambda)$ is increasing (decreasing) then $\lambda\mapsto c(\lambda)$ increasing (decreasing) as well.
\begin{thm}\label{thm:pre_conv_sta} Considering the problem \eqref{eq:def_ulambda} with \eqref{H3c}, \eqref{H4} and $\mathrm{(A1)}$.
\begin{itemize}
\item[(i)] We have the priori estimate $\phi(\lambda) |u_\lambda(x)| + |Du_\lambda(x)| \leq C_H$ for $x\in \Omega_\lambda$.
\item[(ii)] We have $\phi(\lambda) u_\lambda(\cdot)\rightarrow -c(0)$ locally uniformly as $\lambda \rightarrow 0^+$. Furthermore for all $x\in \overline{\Omega}_\lambda$ and $\lambda>0$ we have
\begin{equation}\label{eq:rate1}
\begin{cases}
\left|\phi(\lambda)u_\lambda(x) + c(0)\right| \leq C\left(\phi(\lambda) + |r(\lambda)|\right)\\
\left|\phi(\lambda)u_\lambda(x) + c(\lambda)\right| \leq C\phi(\lambda).
\end{cases}
\end{equation}
As a consequence, whenever $r(\lambda)\neq 0$ there holds
\begin{equation}\label{eq:bound}
    \left|\frac{c(\lambda) - c(0)}{r(\lambda)}\right|  \leq C.
\end{equation}
\item[(iii)] For $x_0\in \overline{\Omega}$ there exists a subsequence $\lambda_j$ and $u,w\in \mathrm{BUC}(\overline{\Omega})\cap \mathrm{Lip}(\Omega)$ such that $u_{\lambda_j}(\cdot) - u_{\lambda_j}(x_0) \rightarrow u(\cdot) $ and $u_{\lambda_j}(\cdot) + \phi(\lambda_j)^{-1}c(0) \rightarrow w(\cdot)$ locally uniformly as $\lambda_j\rightarrow 0$ and $u,w$ solve \eqref{eq:S_0} with $w(x) - u(x) = w(x_0)$.
\end{itemize}
\end{thm}

\begin{proof}[Proof of Theorem \ref{thm:pre_conv_sta}] The priori estimate is clear from the coercivity assumption \eqref{H4}. Fix $x_0\in \Omega$, by Arzel\`a--Ascoli theorem there exists a subsequence $\lambda_j\rightarrow 0^+$, $c\in \mathbb{R}$ and $u$ defined in $\Omega$ such that $\phi(\lambda_j)u_{\lambda_j}(x_0) \rightarrow -c$ and $u_{\lambda_j}(\cdot) - u_\lambda(x_0) \rightarrow u(\cdot)$ locally uniformly as $\lambda_j\rightarrow 0^+$. The case for $w(\cdot)$ can be done in the same manner as well as the relation between $u$ and $w$. It follows that $u\in \mathrm{BUC}(\overline{\Omega})$ and by stability of viscosity solution we have $H(x,Du(x)) = c$ in $\Omega$. Since $u_\lambda(\cdot)$ is Lipschitz, we deduce also that $\phi(\lambda_j)u_{\lambda_j}(x) \rightarrow -c$ for any $x\in \overline{\Omega}$.

We show that $H(x,Du(x)) \geq c$ on $\overline{\Omega}$. Let $\varphi\in \mathrm{C}^1(\overline{\Omega})$ such that $u-\varphi$ has a strict minimum over $\overline{\Omega}$ at $\tilde{x}\in \partial \Omega$, we aim to show that $ H(\tilde{x},D\varphi(\tilde{x})) \geq c$. Let us define 
\begin{equation}\label{eq:u_tilde}
 \tilde{u}_\lambda(x) = (1+r(\lambda))^{-1} u_\lambda\left((1+r(\lambda))x\right), \qquad x\in \overline{\Omega}   
\end{equation}
then
\begin{equation}\label{eq:auxiliary0}
\begin{cases}
\phi(\lambda)(1+r(\lambda)) \tilde{u}_\lambda(x) + H((1+r(\lambda))x, D\tilde{u}_\lambda(x)) \leq 0 &\quad \text{in}\; \Omega,\\
\phi(\lambda)(1+r(\lambda)) \tilde{u}_\lambda(x) + H((1+r(\lambda))x, D\tilde{u}_\lambda(x)) \geq 0 &\quad \text{on}\; \overline{\Omega}.
\end{cases}
\end{equation}
Let us define
\begin{equation*}
    \varphi_\lambda(x) = (1+|r(\lambda)|)\varphi\left(\frac{x}{1+|r(\lambda)|}\right), \qquad x\in (1+|r(\lambda)|)\overline{\Omega}.
\end{equation*}
Note that $D\varphi_\lambda(x) = D\varphi\left((1+|r(\lambda)|)^{-1}x\right)$ for $x\in  (1+|r(\lambda)|\Omega$. We us define
\begin{equation*}
\Phi^\lambda(x,y) = \varphi_\lambda(x) - \tilde{u}_\lambda(y) - \frac{|x-y|^2}{2r(\lambda)^2},\qquad (x,y)\in \left(1+|r(\lambda)|\right)\overline{\Omega}\times \overline{\Omega}.
\end{equation*}
Assume $\Phi^\lambda(x,y)$ has a maximum over $\left(1+|r(\lambda)|\right)\overline{\Omega}\times \overline{\Omega}$ at $(x_\lambda,y_\lambda)$. By definition we have $\Phi^\lambda(x_\lambda,y_\lambda) \geq \Phi^\lambda(y_\lambda,y_\lambda)$, therefore
\begin{equation*}
\varphi_\lambda(x_\lambda) - \frac{|x_\lambda - y_\lambda|^2}{2r(\lambda)^2} \geq \varphi_\lambda(y_\lambda)
\end{equation*}
and thus
\begin{equation*}
    |x_\lambda - y_\lambda| \leq 2r(\lambda) \Vert \varphi\Vert_{L^\infty(\overline{\Omega})}^{1/2}.
\end{equation*}
From that we can assume that $(x_\lambda,y_\lambda)\rightarrow (\overline{x},\overline{x})$ for some $\overline{x}\in \overline{\Omega}$ as $\lambda\rightarrow 0^+$, then
\begin{equation}\label{eq:auxiliary2}
\limsup_{\lambda\rightarrow 0^+}\left(\frac{|x_\lambda - y_\lambda|^2}{2r(\lambda)^2}\right) \leq \limsup_{\lambda\rightarrow 0^+} \left(\varphi_\lambda(x_\lambda) - \varphi_\lambda(y_\lambda)\right) = 0.
\end{equation}
In other words, $|x_\lambda-y_\lambda| = o\left(|r(\lambda)|\right)$. Now $\Phi^\lambda(x_\lambda,y_\lambda) \geq \Phi^\lambda(\tilde{x},\tilde{x})$ gives us
\begin{equation*}
\varphi_\lambda(x_\lambda) - \tilde{u}_\lambda(y_\lambda) - \frac{|x_\lambda - y_\lambda|^2 }{2r(\lambda)^2} \geq \varphi_\lambda(\tilde{x}) - \tilde{u}_\lambda(\tilde{x}).
\end{equation*}
Take $\lambda \rightarrow 0^+$, by \eqref{eq:auxiliary2} we obtain that $u(\tilde{x}) - \varphi(\tilde{x}) \geq u(\overline{x}) - \varphi(\overline{x})$, which implies that $\tilde{x} = \overline{x}$ as $u-\varphi$ has a strict minimum over $\overline{\Omega}$. From $\mathrm{(A1)}$ and $|x_\lambda - y_\lambda| = o(|r(\lambda)|)$, we deduce that $x_\lambda \in (1+|r(\lambda)|)\Omega$. As $y\mapsto \Phi^\lambda(x_\lambda, y)$ has a max at $y_\lambda$, we deduce that
\begin{equation*}
\tilde{u}_\lambda(y) - \left(- \frac{|x_\lambda-y|^2}{2 r(\lambda)^2}    \right)
\end{equation*}
has a minimum at $y_\lambda$, therefore as $\tilde{u}_\lambda$ is Lipschitz with constant $C_H$ we deduce that
\begin{equation}\label{eq:Aug29-1}
    \left|\frac{x_\lambda-y_\lambda}{r(\lambda)^2}\right|\leq C_H, 
\end{equation}
and we can apply the supersolution test for \eqref{eq:auxiliary0} to obtain
\begin{equation}\label{eq:auxiliary4}
\phi(\lambda) (1+r(\lambda))\tilde{u}_\lambda(y_\lambda) + H\left((1+r(\lambda))y_\lambda, \frac{x_\lambda - y_\lambda}{r(\lambda)^2} \right) \geq 0.
\end{equation}
On the other hand, since $x_\lambda \in (1+|r(\lambda)|)\Omega$ as an interior point and $x\mapsto \Phi^\lambda(x, y_\lambda)$ has a max at $x_\lambda$, we deduce that 
\begin{equation}\label{eq:auxiliary5}
D\varphi_\lambda(x_\lambda) = \frac{x_\lambda - y_\lambda}{r(\lambda)^2} \qquad\Longrightarrow\qquad D\varphi\left(\frac{x_\lambda}{1+r(\lambda)}\right) = \frac{x_\lambda - y_\lambda}{r(\lambda)^2}.
\end{equation}
From \eqref{eq:u_tilde},  \eqref{eq:auxiliary4} and \eqref{eq:auxiliary5} we obtain
\begin{equation}\label{eq:aug1}
\phi(\lambda) u_\lambda\left((1+r(\lambda))y_\lambda\right) + H\left((1+r(\lambda))y_\lambda, D\varphi_\lambda(x_\lambda)\right) \geq 0.
\end{equation}
Recall that $\phi(\lambda_j)u_{\lambda_j}(x)\rightarrow -c$ uniformly as $\lambda_j \rightarrow 0^+$ for any $x\in \overline{\Omega}$, we observe that
\begin{align*}
    \left|\phi(\lambda) u_\lambda\left((1+r(\lambda))y_\lambda\right) +c \right| &\leq \big|\phi(\lambda)u_\lambda(\tilde{x})+c\big| +  \phi(\lambda)\big|u_\lambda\left((1+r(\lambda))y_\lambda\right) - u_\lambda\left(\tilde{x}\right)\big|\\
    &\leq \big|\phi(\lambda)u_\lambda(\tilde{x})+c\big| + \phi(\lambda)C_H\big|(y_\lambda - \tilde{x})+r(\lambda)y_\lambda\big|\\ 
    &\leq \big|\phi(\lambda)u_\lambda(\tilde{x})+c\big| + \phi(\lambda)C_H\big|y_\lambda - \tilde{x}\big| +C_H \phi(\lambda)|r(\lambda)|\mathrm{diam}\Omega.
\end{align*}
Let $\lambda \rightarrow 0^+$ along $\lambda_j$ we obtain
\begin{equation}\label{eq:Aug2}
    \lim_{\lambda_j\rightarrow 0^+} \phi(\lambda_j) u_{\lambda_j}\left((1+r(\lambda_j))y_{\lambda_j}\right) = c.
\end{equation}
From \eqref{eq:Aug29-1} and \eqref{H3c} we have (up to subsequences)
\begin{equation}\label{eq:Aug3}
    \lim_{\lambda_j\rightarrow 0^+}H\left((1+r(\lambda_j))y_{\lambda_j}, D\varphi_{\lambda_j}(x_{\lambda_j})\right) = H(\tilde{x}, D\varphi(\tilde{x})).
\end{equation}
From \eqref{eq:aug1}, \eqref{eq:Aug2} and \eqref{eq:Aug3} we deduce that $H(\tilde{x},D\varphi(\tilde{x})) \geq c$. The comparison principle for state-constraint problem gives us the uniqueness of $c$ and furthermore that $c = c(0)$. 

The estimate \eqref{eq:rate1} can be established using comparison principle. We see that $u(x)-\phi(\lambda)^{-1}c(0)  - C$, $u(x)-\phi(\lambda)^{-1}c(0)  + C$ are subsolution and supersolution, respectively, to
\begin{equation}\label{eqn:i}
\begin{cases}
\phi(\lambda) w(x)+H(x,Dw(x)) \leq 0\qquad&\text{in}\;\Omega,\\
\phi(\lambda) w(x)+H(x,Dw(x)) \geq 0\qquad&\text{on}\;\overline{\Omega}.
\end{cases}
\end{equation}
On the other hand, from \eqref{eq:auxiliary0}, the priori estimate $|\phi(\lambda)u_\lambda|\leq C$ and \eqref{H3c} we have $\tilde{u}_\lambda(x) - C\phi(\lambda)^{-1}|r(\lambda)|$, $\tilde{u}_\lambda(x) + C\phi(\lambda)^{-1}|r(\lambda)|$ are subsolution and supersolution, respectively, to \eqref{eqn:i}. Therefore by comparison principle for \eqref{eqn:i} we have 
\begin{equation*}
    \begin{cases}
    u(x) - \phi(\lambda)^{-1}c(0)-C \leq \tilde{u}_\lambda(x) + C|r(\lambda)|\phi(\lambda)^{-1}c(0) ,\\
    u(x) - \phi(\lambda)^{-1}c(0)+C \geq \tilde{u}_\lambda(x) - C|r(\lambda)|\phi(\lambda)^{-1}c(0).
    \end{cases}
\end{equation*}
Therefore $|\phi(\lambda)\tilde{u}_\lambda(x) + c(0)| \leq C\left(\phi(\lambda) + |r(\lambda)|\right)$. The other estimate in \eqref{eq:rate1} is a direct consequence of \eqref{eq:rate0}.
\end{proof}

\begin{rem}\label{rem:Aug29-2} We note that $\tilde{u}_\lambda$ defined as in \eqref{eq:u_tilde} is not necessarily close to $u_\lambda$. In fact, for $x\in \overline{\Omega}$ we have
\begin{equation*}
    \tilde{u}_\lambda(x) - u_\lambda(x)= \frac{u_\lambda((1+r(\lambda))x) - u_\lambda(x)}{1+r(\lambda)} - \frac{r(\lambda)}{1+r(\lambda)} \left(u_\lambda(x)+\frac{c(0)}{\phi(\lambda)}\right)+\frac{r(\lambda)c(0)}{\phi(\lambda)(1+r(\lambda))}.
\end{equation*}
Using \eqref{eq:rate1} and $u_\lambda$ is Lipschitz, we obtain that
\begin{equation}\label{eq:Aug4}
    \left|\tilde{u}_\lambda(x) - u_\lambda(x)\right| \leq 2C(|r(\lambda)||x|) + 2C|r(\lambda)|\left(1+\left|\frac{r(\lambda)}{\varphi(\lambda)}\right|\right)+2\left|\frac{r(\lambda)}{\phi(\lambda)}c(0)\right|.
\end{equation}
Therefore $\tilde{u}_\lambda$ and $u_\lambda$ are close if $\gamma = 0$, and $\left\lbrace\tilde{u}_\lambda+\phi(\lambda)^{-1}c(0)\right\rbrace_{\lambda>0}$ is uniformly bounded in $\lambda>0$ if $\gamma$ is finite (or more generally if $\left|r(\lambda)/\phi(\lambda)\right|$) is bounded, in which case 
\begin{equation}\label{rem:Aug29-3}
    \lim_{\lambda\rightarrow 0^+} \Big(\tilde{u}_\lambda(x) - u_\lambda(x)\Big) = \gamma c(0).
\end{equation}
\end{rem}

As we are working with domains that are smaller or bigger than $\Omega$, we introduce the scaling of measures for convenience. 

\begin{defn}\label{defn:scaledown} For a measure $\sigma$ defined on $(1+r)\overline{\Omega}\times \overline{B}_h$, we define its scaling $\tilde{\sigma 
}$ as a measure on $\overline{\Omega}\times \overline{B}_h$ by
\begin{equation}\label{def_measures}
    \int_{\overline{\Omega}\times\overline{B}_h} f(x,v)\;d\tilde{\sigma}(x,v) = \int_{(1+r)\overline{\Omega}\times\overline{B}_h} f\left(\frac{x}{1+r},v\right)\;d\sigma(x,v).
\end{equation}
\end{defn}

We introduce the following definition for simplicity, as we will deal with mainly approximation from the inside and outside of $\Omega$.

\begin{defn}\label{def:2domains} For $r(\lambda)\geq 0$, we define $\Omega_\lambda^{\pm} = (1\pm r(\lambda))\Omega$. We denote by $c(\lambda)^{\pm}$ and $u_\lambda^{\pm}$, respectively, the additive eigenvalues of $H$ in $(1\pm r(\lambda))\Omega$ and the solutions to the discounted problem \eqref{state-def} on $(1\pm r(\lambda))\Omega$ with discount factor $\delta = \phi(\lambda)$. We let $u_\lambda^-$ and $u_{\lambda}^+$ be solutions to 
\begin{equation*}
    \begin{cases}
    \phi(\lambda) v(x) + H(x,Dv(x)) \leq 0 \quad\text{in}\;\Omega_\lambda,\\
    \phi(\lambda) v(x) + H(x,Dv(x)) \geq 0 \quad\text{on}\;\overline{\Omega}_\lambda;
    \end{cases}
\end{equation*}
with $\Omega_\lambda$ being replaced by $(1-r(\lambda))\Omega$ and $(1+r(\lambda))\Omega$, respectively.
\end{defn}

\section{The first normalization: convergence and a counter example}\label{sec3} In view of Theorems \ref{thm:pre_conv_sta}, it is natural to ask the question if the convergence of $u_\lambda(x) - u_\lambda(x_0)$ holds for the whole sequence as $\lambda \rightarrow 0^+$. The two natural normalization one can study are
\begin{equation}\label{eq:want}
    \left\lbrace u_\lambda(x) + \frac{c(0)}{\phi(\lambda)} \right\rbrace_{\lambda>0} \qquad\text{and}\qquad \left\lbrace u_\lambda(x) + \frac{c(\lambda)}{\phi(\lambda)} \right\rbrace_{\lambda>0}.
\end{equation}
We observe that from Theorem \ref{thm:pre_conv_sta} we have
\begin{equation}\label{eq:bdd}
    \left|u_\lambda(x) + \frac{c(0)}{\phi(\lambda)}\right| \leq C\left(1+\frac{r(\lambda)}{\phi(\lambda)}\right) \qquad\text{and}\qquad \left|u_\lambda(x) + \frac{c(\lambda)}{\phi(\lambda)}\right| \leq C.
\end{equation}
We observe that $u_\lambda(x)+\phi(\lambda)^{-1}c(0)$ is bounded if $\gamma$ defined in \eqref{eq:asm} is finite, or more generally if $|r(\lambda)| = \mathcal{O}(\phi(\lambda))$ as $\lambda\rightarrow 0^+$, while $u_\lambda(x)+\phi(\lambda)^{-1}c(\lambda)$ is bounded even if $\gamma$ is infinite. The following example show a divergence for $u_\lambda(x)+\phi(\lambda)^{-1}c(0)$ when $\gamma = \infty$.

\begin{exa}\label{ex:a} Let us consider $H(x,p) = |p|+x$, $\Omega =(-1,1)$, $\phi(\lambda) = \lambda$ and $r(\lambda) = \lambda^m$ for $\lambda > 0$.   Using the optimal control formula we obtain 
\begin{equation*}
    u_\lambda(x) = \inf_{\alpha(\cdot)} \left(-\int_0^\infty e^{-\lambda s}y(s)\;ds\right) \qquad\text{where}\qquad \begin{cases}
    \dot{y}(s) &= \alpha(s)\in [-1,1]\\
    y(0) &= x.
    \end{cases}
\end{equation*}
Regarding Definition \ref{def:2domains}, we have $c(0) = 1$, $c(\lambda)^{\pm} = 1\pm\lambda^m$ and
\begin{align*}
    u_\lambda(x)^{\pm} + \frac{c(0)}{\lambda} = \frac{1-x}{\lambda}+\frac{e^{-\lambda(1\pm \lambda^m-x)}-1}{\lambda^2} = \mp \lambda^{m-1} + \frac{(1-x\pm \lambda^m)^2}{2}+ \mathcal{O}(\lambda)
\end{align*}
as $\lambda\rightarrow 0^+$, which are convergent only if $m\geq 1$. On the other hand, we have
\begin{equation*}
    u_\lambda(x)^{\pm} + \frac{c(\lambda)^{\pm}}{\lambda} = \frac{(1-x \pm  \lambda^m)^2}{2} + \mathcal{O}(\lambda)
\end{equation*}
as $\lambda\rightarrow 0^+$, which converge to the same limit for all $m\geq 0$. In this example the family $\left\lbrace u_\lambda+\phi(\lambda)^{-1}c(\lambda)\right\rbrace_{\lambda>0}$ still converges even if $\gamma = \infty$. However it is not true in general, as we will prove an example in Section 5.
\end{exa}

\noindent We give a simple proof for the convergence of both families in \eqref{eq:want} when $\gamma = 0$.

\begin{proof}[Proof of Theorem \ref{thm:subcritical}] Let $v_\lambda\in \mathrm{C}(\overline{\Omega})\cap \mathrm{Lip}(\Omega)$ solving
\begin{equation}\label{eq:def_vlambda}
\begin{cases}
\phi(\lambda) v_\lambda(x) + H(x,Dv_\lambda(x)) \leq 0 &\quad\text{in}\;\Omega,\\
\phi(\lambda) v_\lambda(x) + H(x,Dv_\lambda(x)) \geq 0 &\quad\text{on}\;\overline{\Omega}.
\end{cases}
\end{equation}
By Theorem \ref{thm:conv_bdd}, there exists $u^0$ solves \eqref{eq:S_0} such that $v_\lambda(x) + \phi(\lambda)^{-1}c(0) \rightarrow u^0(x)$ uniformly on $\overline{\Omega}$ as $\lambda \rightarrow 0^+$. Define $\tilde{u}_\lambda(x)$ as in \eqref{eq:u_tilde} then $\tilde{u}_\lambda$ solves \eqref{eq:auxiliary0}. Similarly to Theorem \ref{thm:pre_conv_sta} we obtain that $\tilde{u}_\lambda(x) - C\phi(\lambda)^{-1}|r(\lambda)|, \tilde{u}_\lambda(x) + C\phi(\lambda)^{-1}|r(\lambda)|$ are subsolution and supersolution, respectively, to \eqref{eq:def_vlambda}, therefore
\begin{equation*}
    \left|\left(\tilde{u}_\lambda(x) + \frac{c(0)}{\phi(\lambda)}\right) - \left( v_\lambda(x) + \frac{c(0)}{\phi(\lambda)}\right)\right|  \leq C\frac{|r(\lambda)|}{\phi(\lambda)}.
\end{equation*}
Recall \eqref{eq:Aug4}, as $\gamma = 0$ we have $|\tilde{u}_\lambda - u_\lambda|\rightarrow 0$ as $\lambda\rightarrow 0^+$, therefore we deduce that $u_\lambda(x) + \phi(\lambda)^{-1}c(0)\rightarrow u^0(x)$ locally uniformly as $\lambda\rightarrow 0^+$. From \eqref{eq:bound} in Theorem \ref{thm:pre_conv_sta} and $\gamma = 0$ we obtain 
\begin{equation*}
    \lim_{\lambda \rightarrow 0^+}\left( \frac{c(\lambda) - c(0)}{\phi(\lambda)} \right)= 0.
\end{equation*}
and thus $u_\lambda(x) + \phi(\lambda)^{-1}c(\lambda) \rightarrow u^0(x)$ locally uniformly as $\lambda \rightarrow 0^+$.
\end{proof}

\begin{rem} If $\gamma\neq 0$ then in general the solution $v_\lambda$ to \eqref{eq:def_vlambda} and the solution $u_\lambda$ to \eqref{eq:def_ulambda} are not close to each other, as we will see in example \ref{ex:differ}.
\end{rem}

\begin{exa}\label{ex:differ} Let $H(x,p) = |p| - e^{-|x|}$ on $\Omega = (-1,1)$ and $\phi(\lambda)=r(\lambda) =\lambda$. Using the optimal control formula, solutions to \eqref{eq:def_ulambda} (regarding Definition \ref{def:2domains}) are 
\begin{align*}
u^-_\lambda(x) = \frac{e^{-|x|}}{1+\lambda} + \frac{e^{-(1-\lambda^2)+\lambda|x|}}{\lambda(1+\lambda)}, \qquad x\in [-(1-\lambda), (1-\lambda)],\\
u^+_\lambda(x) = \frac{e^{-|x|}}{1+\lambda} + \frac{e^{-(1+\lambda)^2+\lambda|x|}}{\lambda(1+\lambda)}, \qquad x\in [-(1+\lambda), (1+\lambda)].
\end{align*}
On fixed bounded domain, the solution $v_\lambda$ to \eqref{eq:def_vlambda} is given by 
\begin{equation*}
v_\lambda(x) = \frac{e^{-|x|}}{1+\lambda} + \frac{e^{-1-\lambda+\lambda|x|}}{\lambda(1+\lambda)}, \qquad x\in [-1,1].
\end{equation*}
We have $c(0) = -e^{-1}$ and $c(\lambda)^{\pm} = -e^{-1\mp\lambda}$, thus $c^{(1)}_-=c^{(1)}_+ = e^{-1}$ and
\begin{equation*}
    \lim_{\lambda\rightarrow 0^+}\left( \frac{c(\lambda)_- -  c(0)}{-\lambda}\right)  = \lim_{\lambda\rightarrow 0^+}\left( \frac{c(\lambda)_+ - c(0)}{\lambda}\right) = e^{-1}.
\end{equation*}
The maximal solution (in the sense of Theorem \ref{thm:conv_bdd}) on $\Omega$ is given by
\begin{align*}
  u^0(x)= \lim_{\lambda\rightarrow 0^+} \left(v_\lambda(x) + \frac{c(0)}{\lambda}\right)= e^{-|x|} + e^{-1}|x| - e^{-1}, \qquad x\in [-1,1].
\end{align*}
On the other hand, using notation as in Theorem \ref{thm:general} with $\gamma = 1$ we have
\begin{align*}
   u^{-1}= \lim_{\lambda\rightarrow 0^+} \left(u^-_\lambda(x) + \frac{c(0)}{\lambda}\right) &= e^{-|x|} + e^{-1}|x|, \;\;\qquad\qquad x\in [-1,1],\\
    u^{+1} =\lim_{\lambda\rightarrow 0^+}  \left(u^+_\lambda(x) + \frac{c(0)}{\lambda}\right) &= e^{-|x|} + e^{-1}|x|-2e^{-1}, \;\;\quad x\in [-1,1]
\end{align*}
and
\begin{equation*}
     \lim_{\lambda\rightarrow 0^+} \left(u_\lambda(x)^{\pm} + \frac{c(\lambda)_{\pm}}{\lambda}\right)=  u^0(x), \qquad x\in [-1,1].
\end{equation*}
In this example $u^{+1}(\cdot)+u^{-1}(\cdot) = 2u^0(\cdot)$ and $u_\lambda$ and $v_\lambda$ are not close to each other.
\end{exa}

Using the representation formula as in Theorem \ref{thm:conv_bdd}, we show the convergence of $\left\lbrace u_\lambda+\phi(\lambda)^{-1}c(0)\right\rbrace _{\lambda>0}$ when $\gamma$ is finite. This method also recovers the result of Theorem \ref{thm:subcritical}. The following technical lemma is a consequence from $\mathrm{(H4)}$, we give a proof for it in Appendix.

\begin{lem}\label{lem:regu} Assume $L$ satisfies $\mathrm{(H4)}$ then
\begin{align*}
    \frac{L(x,v) - L((1\pm \delta)x,v)}{\delta} \rightarrow (\mp x)\cdot D_xL(x,v) \quad\text{uniformly on}\;\overline{\Omega}\times \overline{B}_h\;\text{as}\;\delta\rightarrow 0^+
\end{align*}
\end{lem}

\begin{proof}[Proof of Theorem \ref{thm:general}] By the reduction step earlier, we may assume that $H$ satisfies \eqref{eq:H&L} for some $h > 0$ and $L \in \mathrm{C}(\overline{\Omega} \times  \overline{B}_h)$. By \eqref{eq:bdd} and $\gamma < \infty$ we have the boundedness of $\left\lbrace u_\lambda(x)+\phi(\lambda)^{-1}c(0)\right\rbrace_{\lambda>0}$. 

Recall Remark \ref{rem:Aug29-2}, let $\tilde{u}_\lambda$ be defined as in \eqref{eq:u_tilde} and $\mathcal{\tilde{U}}$ be the set of accumulation points of $\left\lbrace \tilde{u}_\lambda + \phi(\lambda)^{-1}c(0) \right\rbrace_{\lambda>0}$ in $\mathrm{C}(\overline{U})$ as $\lambda\rightarrow 0^+$. By Theorem \ref{thm:pre_conv_sta} we have $\mathcal{\tilde{U}}$ is nonempty. To show that $\mathcal{\tilde{U}}$ is singleton, we show that if $u,w\in \mathcal{\tilde{U}}$ then $u\equiv w$. 

Assume that there exist $\lambda_j\rightarrow 0$ and $\delta_j \rightarrow 0$ such that $\tilde{u}_{\lambda_j}+\phi(\lambda_j)^{-1}c(0)\rightarrow u$ and $\tilde{u}_{\delta_j}+\phi(\delta_j)^{-1}c(0)\rightarrow w$ locally uniformly as $j \rightarrow \infty$. Let us fix $z\in \Omega$, by Theorem \ref{thm:lambdau} there exists $\mu_{\lambda} \in \mathcal{P}\cap \mathcal{G}'_{z,\phi(\lambda),\Omega_{\lambda}}$ such that  
\begin{equation}\label{eq:xyz}
    \phi(\lambda) u_{\lambda}(z) = \int_{\overline{\Omega}_{\lambda}\times \overline{B}_h} L(x,v)\;d\mu_{\lambda}(x,v) = \min_{\mu\in \mathcal{P}\cap \mathcal{G}'_{z,\phi(\lambda),\Omega_{\lambda}}} \int_{\overline{\Omega}_\lambda\times \overline{B}_h} L(x,v)\;d\mu(x,v).
\end{equation}
Let $\tilde{\mu}_\lambda$ be the measure obtained from $\mu_\lambda$ defined as in Definition \ref{defn:scaledown}, it is clear that $\tilde{\mu}_\lambda$ is a probability measures on $\overline{\Omega}$, therefore the set
\begin{equation}\label{eq:U_*}
    \mathcal{U}_*(z) = \left\lbrace \mu\in \mathcal{P}(\overline{\Omega}\times \overline{B}_h): \tilde{\mu}_\lambda \rightharpoonup \mu\;\text{in measure along some subsequences} \right\rbrace
\end{equation}
is nonempty. By Lemma \ref{lem:stability} we can assume that (up to subsequence) there exists $\mu_0 \in \mathcal{M}_0$ such that $\tilde{\mu}_\lambda\rightharpoonup \mu_0$ in measure. We have $H(x,Dw(x))\leq c(0)$ in $\Omega$, let $w_\lambda(x) = (1+r(\lambda))w\left((1+r(\lambda))^{-1}x\right)$ in $x\in (1+r(\lambda))\overline{\Omega}$ then $w_\lambda(x)\rightarrow w(x)$ pointwise s $\lambda\rightarrow 0^+$ and
\begin{equation*}
\phi(\lambda)w_\lambda(x)+H_{L\left(\frac{x}{1+r(\lambda)},v\right)+\phi(\lambda) w_\lambda(x)+c(0)}(x,Dw_\lambda(x)) \leq 0 \qquad\text{in}\;(1+r(\lambda))\Omega.
\end{equation*}
By definition we obtain
\begin{equation*}
\left( L\left(\frac{x}{1+r(\lambda)},v\right)+\phi(\lambda) w_{\lambda}(x)+c(0), w_{\lambda}(x) \right)\in \mathcal{F}_{\phi(\lambda),\Omega_{\lambda}}
\end{equation*}
and therefore
\begin{equation*}
    \left\langle \mu_{\lambda}, L\left(\frac{x}{1+r(\lambda)},v\right)+\phi(\lambda) w_{\lambda}(x)- \phi(\lambda) w_{\lambda}(z)+c(0)\right\rangle \geq 0.
\end{equation*}
In other words, we have
\begin{equation*}
    \left\langle \mu_\lambda, L\left(\frac{x}{1+r(\lambda)},v\right)\right\rangle+\phi(\lambda)(1+r(\lambda))\left\langle \mu_\lambda, w\left(\frac{x}{1+r(\lambda)}\right)\right\rangle + c(0) \geq \phi(\lambda) w_{\lambda}(z).
\end{equation*}
Combine with $-\langle \mu_\lambda, L(x,v)\rangle + \phi(\lambda) u_\lambda(z) = 0
$ from \eqref{eq:xyz} we obtain
\begin{align*}
     \left\langle \mu_\lambda, L\left(\frac{x}{1+r(\lambda)},v\right) -L(x,v)\right\rangle  &+ \phi(\lambda)(1+r(\lambda)) \big\langle\tilde{\mu}_\lambda,w(x)\big\rangle \\
     &+ \phi(\lambda) u_\lambda(z) + c(0) \geq \phi(\lambda) w_\lambda(z).
\end{align*}
Dividing both sides by $\phi(\lambda)$ we deduce that
\begin{equation*}
    \frac{r(\lambda)}{\phi(\lambda)}\left\langle \tilde{\mu}_{\lambda},\frac{L(x,v) - L((1+r(\lambda))x,v)}{r(\lambda)} \right\rangle  + (1+r(\lambda))\left\langle \tilde{\mu}_{\lambda}, w\right\rangle  + \left( u_{\lambda}(z)+\frac{c(0)}{\phi(\lambda)}\right)\geq w_{\lambda}(z).
\end{equation*}
Since $\tilde{\mu}_{\lambda_j}\rightharpoonup \mu_0$ in measure, using Lemma \ref{lem:regu} we deduce that
\begin{equation}\label{eqn:ge1}
\gamma\left\langle \mu_0, (-x)\cdot D_xL(x,v) \right\rangle +\langle \mu_0, w\rangle + \big(u(z)-\gamma c(0)\big) \geq w(z),
\end{equation}
where $u_\lambda(z)+\phi(\lambda_j)^{-1}c(0) \rightarrow \left(u(z)-\gamma c(0)\right)$ comes from $\tilde{u}_\lambda(z)+\phi(\lambda_j)^{-1}c(0) \rightarrow u(z)$ and \eqref{rem:Aug29-3} in Remark \ref{rem:Aug29-2}. On the other hand, from \eqref{eq:auxiliary0} we have
\begin{equation*}
    \phi(\lambda)\tilde{u}_\lambda(x) + H\big((1+r(\lambda))x,D\tilde{u}_\lambda(x)\big) \leq 0 \qquad\text{in}\;\Omega
\end{equation*}
In other words, we have
\begin{equation*}
    L\big((1+r(\lambda))x,v\big) - \phi(\lambda)(1+r(\lambda))\tilde{u}_\lambda(x) \in \mathcal{F}_{0,\Omega}
\end{equation*}
and thus
\begin{equation*}
    \left\langle \mu,  L\big((1+r(\lambda))x,v\big) - \phi(\lambda)(1+r(\lambda))\tilde{u}_\lambda(x)\right\rangle \geq 0 \qquad\text{for all}\;\mu\in \mathcal{M}_0.
\end{equation*}
Recall that $- \langle \mu, L(x,v)\rangle = c(0)$ for all $\mu\in \mathcal{M}_0$, we have
\begin{equation*}
    \frac{r(\lambda)}{\phi(\lambda)}\left\langle \mu, \frac{L\big((1+r(\lambda))x,v\big) - L(x,v)}{r(\lambda)} \right\rangle \geq (1+r(\lambda)) \left\langle \mu, \tilde{u}_\lambda(x)+\frac{c(0)}{\phi(\lambda)}\right\rangle - \frac{r(\lambda)}{\phi(\lambda)}c(0) 
\end{equation*}
for all $\mu\in \mathcal{M}_0$. Let $\lambda=\delta_j$ then as $j\rightarrow \infty$ we have $\gamma\langle \mu, x\cdot D_xL(x,v)\rangle \geq \langle \mu, w\rangle - \gamma c(0)$, i.e.,
\begin{equation}\label{eq:s1}
    \gamma\langle \mu, (-x)\cdot D_xL(x,v)\rangle +\langle \mu, w\rangle -\gamma c(0) \leq 0, \qquad\text{for all}\;\mu\in \mathcal{M}_0.
\end{equation}
From \eqref{eqn:ge1} and \eqref{eq:s1} we deduce that $u(z) \geq w(z)$. Since $z\in \Omega$ arbitrarily we have $u\geq w$ and similarly $u\leq w$, thus $u\equiv w$ and we have the uniform convergence for the full sequence 
\begin{equation*}
    \lim_{\lambda\to 0}\left(\tilde{u}_\lambda(x) + \frac{c(0)}{\phi(\lambda)}\right)
\end{equation*}
Denote this limit as $\tilde{u}^\gamma$, then from Remark \ref{rem:Aug29-2} we have $u_\lambda+\phi(\lambda)^{-1}c(0)\rightarrow u^\gamma = \tilde{u}^\gamma- \gamma c(0)$ locally uniformly in $\Omega$ as $\lambda\rightarrow0^+$. Clearly $u^\gamma \in \mathcal{E}^\gamma$ thanks to \eqref{eq:s1}. If $v\in \mathcal{E}^\gamma$ then since $\mu_0\in \mathcal{M}_0$, we can establish \eqref{eqn:ge1} with $w$ being replaced by $v$ to obtain $u^\gamma\geq v$, hence $u^\gamma = \sup \mathcal{E}^\gamma$.
\end{proof}

\begin{cor}\label{cor:my} For any $\mu \in \mathcal{U}_*(z)$ there holds $\gamma\big\langle \mu, (-x)\cdot D_xL(x,v)\big\rangle + \big\langle \mu, u^\gamma   \big\rangle = 0$. 
\end{cor}

\begin{lem}\label{rem: on L} For any $\mu\in \mathcal{M}_0$ there holds $\langle \mu, (-x)\cdot D_xL(x,v)\rangle \geq 0$.
\end{lem}
\begin{proof}[Proof of Lemma \ref{rem: on L}] For $\mu\in \mathcal{M}_0$ and $0<\lambda \ll 1$ we define $\mu_\lambda$ by
\begin{equation*}
    \langle \mu_\lambda, f(x,v) \rangle := \langle \mu, f((1-\lambda)x,v)\rangle, \qquad\text{for}\;f\in \mathrm{C}(\overline{\Omega}\times \overline{B}_h).
\end{equation*}
It is easy to see that $\mu_\lambda$ is a probability measure on $\overline{\Omega}\times \overline{B}_h$. Furthermore $\mu_\lambda\in \mathcal{G}_{0,\Omega}'$ as well. In fact, if $f\in \mathcal{G}_{0,\Omega}$ then there exists $u\in \mathrm{C}(\overline{\Omega})$ such that $H_f(x,Du(x))\leq 0$ in $\Omega$. It is clear that $H_{f((1-\lambda)x,v)}(x,D\tilde{u}(x))\leq 0$ in $\Omega$ as well where $\tilde{u}(x) = (1-\lambda)^{-1}u((1-\lambda)x)$, therefore $f((1-\lambda)x,v)\in \mathcal{G}_{0,\Omega}$, hence
\begin{equation*}
    \left\langle \mu_\lambda, f(x,v) \right\rangle = \left\langle \mu, f((1-\lambda)x,v) \right\rangle \geq 0.
\end{equation*}
As $\mu_\lambda\in \mathcal{P}\cap \mathcal{G}'_{0,\Omega}$, we deduce that
\begin{equation*}
\left\langle \mu, L((1-\lambda)x,v) \right\rangle =     \left\langle \mu_\lambda, L(x,v) \right\rangle \geq \left\langle \mu, L(x,v) \right\rangle.
\end{equation*}
Let $\lambda\rightarrow 0^+$ we deduce that $\langle \mu, (-x)\cdot D_xL(x,v)\rangle \geq 0$.
\end{proof}

\begin{rem} If the domain is periodic then by translation invariant, we can get an invariant for Mather measures as 
\begin{equation*}
    \langle \mu, (-x)\cdot D_xL(x,v)\rangle = 0 \qquad\text{for all}\; \mu\in \mathcal{M}_0.
\end{equation*}
We refer the reader's to \cite{biryuk_introduction_2010} for more properties like this in the case of periodic domain. In our setting, it is natural to expect a similar invariant holds. Indeed, it is interesting that
\begin{equation*}
    \begin{cases}
    \langle \mu,(-x)\cdot D_xL(x,v)\rangle = \text{constant}\\ 
    \text{for all}\; \mu\in \mathcal{M}_0
    \end{cases} \qquad\Longleftrightarrow\qquad \lambda\mapsto c(\lambda)\;\text{is differentiable at}\;\lambda = 0
\end{equation*}
and if that is the case, the constant in the above is $c'(0)$, the derivative of the map $\lambda\to c(\lambda)$ at $\lambda = 0$. For instance, in Example \ref{ex:differ} we have $c'(0) = e^{-1}$.
\end{rem}

\begin{lem}\label{lem:stability} We have $\mathcal{U}_*(z)\subseteq \mathcal{M}_0$ where $\mathcal{U}_*(z)$ is defined as in \eqref{eq:U_*}.
\end{lem}
\begin{proof} Assume $\tilde{\mu}_{\lambda_j}\rightharpoonup \mu_0$. From \eqref{eq:xyz} it is clear that $-c(0) = \langle \mu_0, L\rangle$. For $f\in \mathcal{G}_{0,\Omega}$ there exists $u\in \mathrm{C}(\overline{\Omega})$ such that $H_f(x,Du(x))\leq 0$ in $\Omega$. Let us define $\tilde{u}$ as in \eqref{eq:u_tilde}, we have
\begin{equation*}
 \phi(\lambda)\tilde{u}(x) + H_{f\left(\frac{x}{1+r(\lambda)},v\right)+\phi(\lambda)\tilde{u}(x)}(x,D\tilde{u}(x)(x))\leq 0\quad\text{in}\;(1+r(\lambda))\Omega.
\end{equation*}
By definition we deduce that
\begin{equation*}
\big\langle \tilde{\mu}_\lambda, f(x,v) +\phi(\lambda)(1+r(\lambda))\left(u - u(z)\right) \big\rangle=     \left\langle \mu_\lambda, f\left(\frac{x}{1+r(\lambda)},v\right) + \phi(\lambda)\left(\tilde{u} - \tilde{u}(z)\right) \right\rangle \geq 0
\end{equation*}
Let $\lambda\rightarrow 0^+$ along $\lambda_j$ we deduce that $\langle \mu_0, f \rangle \geq 0$, hence $\mu_0\in \mathcal{M}_0$.
\end{proof}

\begin{proof}[Proof of Corollary \ref{cor:mycor}] From \eqref{eq:s1} we have 
\begin{equation*}
    \begin{cases}
    \alpha\langle \mu, (-x)\cdot D_xL(x,v)\rangle +\langle \mu, u^\alpha\rangle \leq 0 \\
    \beta\langle \mu, (-x)\cdot D_xL(x,v)\rangle +\langle \mu, u^\beta\rangle \leq 0 
    \end{cases}
\end{equation*}
for all $\mu\in \mathcal{M}_0$. If $\theta = \beta- \alpha \geq 0$ then
\begin{equation*}
    \theta \left\langle\mu, (-x)\cdot D_xL(x,v)\right\rangle + \alpha \left\langle\mu, (-x)\cdot D_xL(x,v)\right\rangle + \big\langle \mu, u^\beta\big\rangle \leq 0
\end{equation*}
for all $\mu\in \mathcal{M}_0$. Since $\theta \left\langle\mu, (-x)\cdot D_xL(x,v)\right\rangle\geq 0$, we have $u^\beta \in \mathcal{E}^\alpha$, therefore $u^\beta\leq u^\alpha$. Denote $\gamma = (1-\lambda)\alpha+\lambda \beta$ for $\lambda\in (0,1)$, we have
\begin{equation*}
    \gamma\left\langle \mu, (-x)\cdot D_xL(x,v)\right\rangle + \left\langle\mu, (1-\lambda) u^\alpha +\lambda u^\beta \right\rangle  \leq 0, \qquad\text{for all}\;\mu\in \mathcal{M}_0.
\end{equation*}
By the convexity of $H$ we see that $u = (1-\lambda)u^\alpha+\lambda u^\beta$ belongs to $\mathcal{E}^{(1-\lambda)\alpha+\lambda \beta}$, therefore $(1-\lambda)u^\alpha+\lambda u^\beta \leq u^{(1-\lambda)\alpha+\lambda \beta}$.
\end{proof}

\section{The asymptotic expansion of the eigenvalue}\label{sec4}
\subsection{The expansion at zero}
In this section, we want to study the asymptotic expansion of $c(\lambda)$ as $\lambda\rightarrow 0^+$. If the following limit exist
\begin{equation}\label{lim2}
c^{(1)} = \lim_{\lambda\rightarrow 0^+}\left(\frac{c(\lambda) - c(0)}{r(\lambda)}\right)
\end{equation}
then heuristically we have $c(\lambda) = c(0) + c^{(1)}r(\lambda) + o(r(\lambda))$ as $\lambda \rightarrow 0^+$. 

\begin{rem} The dependence of $\lambda$ and the eigenvalue should be $c(r(\lambda))$ but in fact $c^{(1)}$ is independent of $r(\lambda)$ if it exists. Indeed, assume $\lambda\mapsto c(\lambda) = c(r(\lambda))$ is differentiable at $\lambda = 0$, for any $\mu\in \mathcal{M}_0$ by scaling into a measure $\mu_\lambda$ on $(1+r(\lambda))\Omega$, we can show that $-c(r(\lambda)) \leq \langle \mu_\lambda,L\rangle$, hence
\begin{equation*}
    \big\langle \mu,L((1+r(\lambda))x,v)\big\rangle + c(r(\lambda)) \geq 0 
\end{equation*}
for all $\lambda$ with an equality at $\lambda = r(\lambda) = 0$, therefore
\begin{equation*}
    c^{(1)} = c'(0) = \langle \mu, (-x)\cdot D_xL\rangle.
\end{equation*}
Thus we can simply write $c(\lambda)$ for simplicity and we can simply choose $r(\lambda) = \pm\lambda$.
\end{rem}

We will show that \eqref{lim2} holds when $\lambda\mapsto r(\lambda)$ does not change its sign around $0$, provided that \eqref{H3c} is satisfied. Example \ref{ex:10} shows that it can be divergent if \eqref{H3c} is violated and Example \ref{ex:differ} shows that in general it is not zero. 

\begin{exa}\label{ex:10} Let $H(x,p) = |p|-\sqrt{1-|x|}$ for $(x,p)\in [-1,1]\times \mathbb{R}$. Let $r(\lambda) = \lambda \in (0,1)$, the $c(\lambda) = -\sqrt{\lambda}$ and the limit \eqref{lim2} does not exist.
\end{exa}

\noindent
Let $\nu_\lambda$ be measures in $\mathcal{P}\cap \mathcal{G}'_{0,\Omega_\lambda}$ such that
\begin{equation}\label{eq:limit-p0}
    -c(\lambda) = \int_{\overline{\Omega}_\lambda\times \overline{B}_h} L\left(x,v\right)\;d\nu_\lambda(x,v) = \min_{\nu \in \mathcal{P}\cap \mathcal{G}'_{0,\Omega_\lambda}}\int_{\overline{\Omega}_\lambda\times \overline{B}_h} L\left(x,v\right)\;d\nu(x,v).
\end{equation}
Let $\tilde{\nu}_\lambda$ be the corresponding measures on $\overline{\Omega}$ after scaling from $\nu_\lambda$ as in Definitions \ref{defn:scaledown}, it is easy to see that $\tilde{\nu}_\lambda$ is still a probability measure on $\overline{\Omega}$, thus by compactness the set of weak limit points $\mathcal{V}_*$ of $ \left\lbrace\tilde{\nu}_\lambda\right\rbrace_{\lambda>0}$ is nonempty.

\begin{lem}\label{lem:inM0} We have $\mathcal{V}_*\subseteq \mathcal{M}_0$. 
\end{lem}

\begin{proof}[Proof of Lemma \ref{lem:inM0}] From \eqref{eq:limit-p0} we have
\begin{equation*}
    -c(\lambda) = \int_{\overline{\Omega}\times \overline{B}_h} L\big(\left(1+r(\lambda)\right)x,v\big)\;d\tilde{\nu}_\lambda(x,v).
\end{equation*}
Assume $\tilde{\nu}_{\lambda_j}\rightharpoonup \nu_0$, then since $L\big(\left(1+r(\lambda)\right)x,v\big)\rightarrow L(x,v)$ uniformly as $\lambda\rightarrow 0^+$, we deduce that $-c(0) = \langle \nu_0,L\rangle$. Let $f\in \mathcal{G}_{0,\Omega}$, there exists $u\in \mathrm{C}(\overline{\Omega})$ such that $H_f(x,Du(x))\leq 0$ in $\Omega$. Let us define $\tilde{u}$ as in \eqref{eq:u_tilde}, then
\begin{equation*}
 H_{f_\lambda}(x,D\tilde{u}(x))\leq 0\quad\text{in}\;\Omega_\lambda
\end{equation*}
where $f_\lambda(x,v) = f\left(\frac{x}{1+r(\lambda)},v\right)$. By definition of $\nu_\lambda$ we have
\begin{equation*}
    \left\langle \tilde{\nu}_\lambda, f(x,v) \right\rangle=  \left\langle \nu_\lambda, f_\lambda(x,v) \right\rangle \geq 0.
\end{equation*}
Let $\lambda_j\rightarrow 0^+$ we deduce that $\langle \nu_0, f \rangle \geq 0$, thus $\nu_0\in \mathcal{M}_0$.
\end{proof}

\begin{proof}[Proof of Theorem \ref{thm:limit}] Let us consider the case $r(\lambda)\geq 0$. Let $w_\lambda$ be a solution to \eqref{eq:cell3} and $\tilde{w}_\lambda$ be its scaling as in \eqref{eq:u_tilde}, then
\begin{equation*}
    H_{L((1+r(\lambda))x,v) +c(\lambda)}\big(x,D\tilde{w}_\lambda(x)\big) \leq 0 \qquad\text{in}\;\Omega.
\end{equation*}
Therefore $L((1+r(\lambda))x,v) +c(\lambda)\in \mathcal{F}_{0,\Omega}$, thus 
\begin{equation}\label{e:nice}
    \Big\langle \mu, L((1+r(\lambda))x,v) + c(\lambda)\Big\rangle \geq 0
\end{equation}
for any $\mu\in \mathcal{M}_0$. Using the fact that $-\langle \mu, L(x,v)\rangle = c_0$ we deduce that
\begin{equation*}
    \Big\langle \mu, L((1+r(\lambda))x,v) - L(x,v)\Big\rangle +(c(\lambda) - c(0)) \geq 0.
\end{equation*}
Thus if $r(\lambda)>0$ then for all $\mu\in \mathcal{M}_0$ we have
\begin{equation*}
    \left\langle\mu, \frac{L((1+r(\lambda))x,v) - L(x,v)}{r(\lambda)} \right\rangle + \left(\frac{c(
    \lambda)-c(0)}{r(\lambda)}\right) \geq 0.
\end{equation*}
Using \eqref{eq:bound} and the fact that $r(\lambda)$ is not identically zero near $0$, as $\lambda\rightarrow 0^+$ we have
\begin{equation}\label{eq:limit-p3}
    \langle\mu, x\cdot D_xL(x,v)\rangle + \liminf_{\lambda\rightarrow 0^+}\left(\frac{c(\lambda)-c(0)}{r(\lambda)}\right) \geq 0 \qquad\text{for all}\; \mu \in \mathcal{M}_0.
\end{equation}
Let $\lambda_j\rightarrow 0^+$ be the subsequence such that
\begin{equation*}
    \limsup_{\lambda\rightarrow 0^+}\left(\frac{c(\lambda)-c(0)}{r(\lambda)}\right)  =\lim_{j\rightarrow \infty}\left(\frac{c(\lambda_j)-c(0)}{r(\lambda_j)}\right).
\end{equation*}
For simplicity we can assume that (up to subsequence) $\tilde{\nu}_{\lambda_j}\rightharpoonup \nu_0$ and $\nu_0\in \mathcal{M}_0$. Let $w$ be a solution to \eqref{eq:S_0}, then $\tilde{w}(x) = (1+r(\lambda))w\left((1+r(\lambda))^{-1}x\right)$ solves
\begin{equation*}
    H_{L\left(\frac{x}{1+r(\lambda)},v \right)+c(0)}\left(x,D\tilde{w}(x)\right) \leq 0 \qquad\text{in}\;(1+r(\lambda))\Omega.
\end{equation*}
As $v_\lambda\in \mathcal{P}\cap \mathcal{G}'_{0,\Omega_\lambda}$ and $\langle \nu_\lambda, L\rangle = -c(\lambda)$, we obtain that
\begin{equation*}
    \left\langle \nu_\lambda, L\left(\frac{x}{1+r(\lambda)},v\right) - L\left(x,v\right)\right\rangle -c(\lambda) + c(0) \geq 0.
\end{equation*}
By definition of $\tilde{\nu}_\lambda$, it is equivalent to
\begin{equation}\label{e:nice2}
    \Big\langle \tilde{\nu}_\lambda, L\left(x,v\right) - L\left((1+r(\lambda))x,v\right)\Big\rangle\geq c(\lambda)-c(0).
\end{equation}
As $r(\lambda_j)\geq 0$, let $\lambda_j\rightarrow 0^+$ we obtain 
\begin{equation}\label{eq:limit-p9}
    \left\langle \nu_0, (-x)\cdot D_xL(x,v)\right\rangle \geq  \limsup_{\lambda\rightarrow 0^+} \left(\frac{c(\lambda) - c(0)}{r(\lambda)}\right).
\end{equation}
In \eqref{eq:limit-p3}, take $\mu = \nu_0\in \mathcal{M}_0$ and together with \eqref{eq:limit-p9} we conclude that
\begin{equation*}
    \lim_{\lambda\rightarrow 0^+}\left(\frac{c(\lambda)-c(0)}{r(\lambda)}\right)  =  \left\langle \nu_0,(-x)\cdot D_xL(x,v)\right\rangle = \sup_{\mu\in \mathcal{M}_0} \left\langle \mu,(-x)\cdot D_xL(x,v)\right\rangle.
\end{equation*}
Similarly, if  $r(\lambda)\leq 0$ as $\lambda\rightarrow 0^+$ then 
    \begin{equation*}
        \lim_{\lambda\rightarrow 0^+} \left(\frac{c(\lambda)-c(0)}{r(\lambda)}\right) = \min_{\mu\in \mathcal{M}_0} \left\langle \mu,(-x)\cdot D_xL(x,v)\right\rangle.
    \end{equation*}
For an oscillating $r(\lambda)$ such that neither $r^-(\lambda) = \min\{0,r(\lambda)\}$ nor $r^+(\lambda) = \max\{0,r(\lambda)\}$ is identical to zero as $\lambda\rightarrow 0^+$, by applying the previous results we have (we consider the limit $(c(\lambda) - c(0))/r(\lambda)$ along subsequences where $r(\lambda)\neq 0$) 
\begin{align*}
    &\lim_{\substack{\lambda\rightarrow 0^+\\ r(\lambda) > 0}} \left(\frac{c(\lambda) - c(0)}{r(\lambda)}\right) = \max_{\mu\in \mathcal{M}_0} \left\langle \mu, (-x)\cdot D_xL(x,v)\right\rangle = c^{(1)}_+,\\
    &\lim_{\substack{\lambda\rightarrow 0^+\\ r(\lambda) < 0}} \left(\frac{c(\lambda) - c(0)}{r(\lambda)}\right) = \min_{\mu\in \mathcal{M}_0} \left\langle \mu, (-x)\cdot D_xL(x,v)\right\rangle= c^{(1)}_- .
\end{align*}
For any given subsequence $\lambda_j\rightarrow 0^+$ along which $(c(\lambda_j)-c(0))/r(\lambda_j)$ converges, by decomposing $\lambda_j$ into subsequences where $r(\lambda_j)>0$ and $r(\lambda_j)<0$ respectively, we see that $(c(\lambda_j)-c(0))/r(\lambda_j)$ can only converge either to $c^{(1)}_+$ or $c^{(1)}_-$, and therefore we obtain the conclusion of the theorem.
\end{proof}

\begin{proof}[Proof of Corollary \ref{cor:Aug30-1}] If $\left\langle \mu,(-x)\cdot D_xL(x,v)\right\rangle = c^{(1)}$ for all $\mu\in \mathcal{M}_0$ then from Theorem \ref{thm:general} with $\gamma\in \mathbb{R}$ we have $\left\langle \mu, u^\gamma + \gamma c^{(1)}  \right\rangle \leq 0$ for all $\mu\in \mathcal{M}_0$, thus $u^\gamma + \gamma c^{(1)} \in \mathcal{E}$ and hence $u^\gamma + \gamma c_{(1)} \leq u^0$. On the other hand, $\left\langle\mu, u^0\right\rangle = \gamma c^{(1)} + \left\langle \mu, u^0-\gamma c^{(1)} \right\rangle \leq 0$ for all $\mu\in \mathcal{M}_0$, therefore
\begin{equation*}
    \gamma\langle\mu, (-x)\cdot D_xL(x,v)\rangle + \left\langle\mu, u^0-\gamma c^{(1)}  \right\rangle  \leq 0 \qquad\text{for all}\;\mu\in \mathcal{M}_0.
\end{equation*}
Thus $u^0 - \gamma c^{(1)} \in \mathcal{E}^\gamma$, hence $u^0 - \gamma c^{(1)} \leq u^\gamma$.
\end{proof}

   \begin{rem}\label{remark:Aubry}  Here are some examples where $c^{(1)}_- = c^{(1)}_+ = c^{(1)}$.
\begin{itemize}
    \item[(i)] If $H(x,p) = H(p)+V(x)$ with $x\cdot \nabla V(x) \leq 0$  for all $x\in \Omega$. Indeed, Lemma \ref{rem: on L} says that $\langle\mu, x\cdot \nabla V(x)\rangle \geq 0$ for all $\mu\in \mathcal{M}_0$, thus in this case we have $\langle \mu, (-x)\cdot D_xL(x,v)\rangle = 0$ for all $\mu\in \mathcal{M}_0$, hence $c^{(1)}_- = c^{(1)}_+ = 0$. 
    \item[(ii)] If the Aubry set $\mathcal{A}$ of $H$ is compactly supported in $\Omega$, then by Theorem \ref{thm:eigenvalue} we have $c(\lambda) = c(0)$ for all $\lambda>0$ small enough, therefore $c^{(1)}=0$.
    \item[(iii)] Recall from Example \ref{ex:differ} that if $H(x,p) = |p| - e^{-|x|}$ and $\Omega = (-1,1)$ then $c^{(1)}_+ = c^{(1)}_- =c^{(1)} = e^{-1}$.
\end{itemize}
\end{rem}

We state the following lemma concerning properties of limits of minimizing measures on $\Omega$, which will be used to prove Corollary \ref{cor:equala}.

\begin{lem}\label{lem:sigma0} 
Let $v_\lambda\in \mathrm{C}(\overline{\Omega})$ be the solution to \eqref{eq:def_vlambda}. For $z\in \Omega$, let $\sigma_\lambda\in \mathcal{P}\cap \mathcal{G}'_{z,\phi(\lambda),\Omega}$ be the minimizing measure such that $\phi(\lambda)v_\lambda(z) = \langle \sigma_\lambda,L\rangle$. Let us define 
\begin{equation*}
    \mathcal{U}_0(z) = \left\lbrace \mu\in \mathcal{P}(\overline{\Omega}\times \overline{B}_h): \sigma_\lambda\rightharpoonup \mu\;\text{in measures along some subsequences}\right\rbrace
\end{equation*}
then 
\begin{itemize}
    \item[(i)] $\langle \sigma_0, u^0\rangle =0$ for all $\sigma_0\in \mathcal{U}_0(z)$.
    \item[(ii)] $ u^0(z)  \geq u^\gamma(z) + \gamma\big\langle \sigma_0,(-x)\cdot D_xL(x,v) \big\rangle$ for all $\sigma_0\in \mathcal{U}_0(z)$ and $\gamma\in \R$.
\end{itemize}
\end{lem}
\begin{proof}[Proof of Lemma \ref{lem:sigma0}] It is clear that $\mathcal{U}_0(z)\subset \mathcal{M}_0$.
\begin{itemize}
    \item[(i)] For any $w$ solves \eqref{eq:S_0} we have
\begin{equation}\label{june19-p1}
    \Big\langle \sigma_\lambda, L(x,v) + c(0) +  \phi(\lambda)w(x) - \phi(\lambda)w(z)\Big\rangle \geq 0.
    \end{equation}
Since $v_\lambda+\varphi(\lambda)^{-1}c(0)\rightarrow u^0$ by Theorem \ref{thm:conv_bdd}, from \eqref{june19-p1} we deduce that 
\begin{equation}\label{Sep2-1}
    u^0(z) + \langle \sigma_0, w\rangle \geq w(z)
\end{equation}
for some $\sigma_0\in \mathcal{U}_0(z)$. Let $w = u^0$ we obtain $\langle \sigma_0, u^0\rangle \geq 0$, thus $\langle \sigma_0, u^0\rangle = 0$ since $\langle \mu, u^0\rangle\leq 0$ for all $\mu\in \mathcal{M}_0$. 
\item[(ii)] To connect $u^0$ with $u^\gamma$, we use the approximation $u_\lambda$ on $(1+r(\lambda))\Omega$. Recall that after scaling $\tilde{u}_\lambda(x) = (1+r(\lambda))^{-1}u_\lambda\big((1+r(\lambda))x\big)$ for $x\in \overline{\Omega}$ we have
\begin{equation*}
    L((1+r(\lambda))x,v) - \phi(\lambda)r(\lambda)\tilde{u}_\lambda(x) \in \mathcal{F}_{z,\phi(\lambda),\Omega}.
\end{equation*}
Recall the definition of $\sigma_\lambda\in \mathcal{P}\cap \mathcal{G}'_{z,\phi(\lambda),\Omega}$ from Lemma \ref{lem:sigma0}, we have
\begin{equation*}
    \big\langle \sigma_\lambda, L((1+r(\lambda))x,v) - \phi(\lambda)r(\lambda)\tilde{u}_\lambda - \phi(\lambda)\tilde{u}_\lambda(z) \big\rangle \geq 0.
\end{equation*}
Using $-\langle \sigma_\lambda, L(x,v)\rangle + \phi(\lambda)v_\lambda(z)= 0$ where $v_\lambda$ solves \eqref{eq:def_vlambda} we obtain
\begin{equation*}
    \frac{r(\lambda)}{\phi(\lambda)}\left\langle \sigma_\lambda, \frac{L((1+r(\lambda))x,v) - L(x,v)}{r(\lambda)}\right\rangle + v_\lambda(z) - r(\lambda)\langle\sigma_\lambda,\tilde{u}_\lambda\rangle \geq \tilde{u}_\lambda(z).
\end{equation*}
Taking into account the normalization, we deduce that
\begin{align*}
    &\frac{r(\lambda)}{\phi(\lambda)}\left\langle \sigma_\lambda, \frac{L((1+r(\lambda))x,v) - L(x,v)}{r(\lambda)}\right\rangle +\left( v_\lambda(z) + \frac{c(0)}{\phi(\lambda)}\right) \\
    &\qquad\qquad\qquad - r(\lambda)\left\langle\sigma_\lambda,\tilde{u}_\lambda+\frac{c(0)}{\phi(\lambda)}\right\rangle \geq \left(\tilde{u}_\lambda(z)+ \frac{c(0)}{\phi(\lambda)}\right)-\frac{r(\lambda)}{\phi(\lambda)}c(0).
\end{align*}
Assume $\sigma_\lambda\rightharpoonup \sigma_0$ for some $\sigma_0\in \mathcal{U}_0(z)$, then as $\lambda\to 0$ we have
\begin{equation*}
    \gamma\langle \sigma_0,(+x)\cdot D_xL(x,v) \rangle +u^0(z)  \geq u^\gamma(z)
\end{equation*}
and thus the conclusion $u^0(z)\geq  u^\gamma(z)+ \gamma \left\langle \sigma_0, (-x)\cdot D_xL(x,v) \right\rangle$ follows.
\end{itemize}
\end{proof}

\begin{proof}[Proof of Corollary \ref{cor:equala}] If $\gamma > 0$ then from Lemma \ref{lem:sigma0} there exists $\sigma_0\in \mathcal{U}_0(z)$ such that 
\begin{equation*}
    0 = u^0(z) - u^\gamma(z) \geq \gamma \langle \sigma_0, (-x)\cdot D_xL(x,v)\rangle \geq\gamma c^{(1)}_- \geq 0
\end{equation*}
and thus $c^{(1)}_- = 0$.
\end{proof}

\subsection{The additive eigenvalues as a function}
A natural question that comes from Theorem \ref{thm:limit} is when do we have the invariant \begin{equation*}
    \langle \mu, (-x)\cdot D_xL(x,v)\rangle = c^{(1)}
\end{equation*}
for all $\mu\in \mathcal{M}_0$? In other words, when is the map $\lambda\mapsto c(\lambda)$ is differentiable at $\lambda = 0$? We can indeed study the map $\lambda\mapsto c(\lambda)$ on an open interval $I$ including zero, and ask the question at what point $\lambda$ where $c'(\lambda)$ exists. It is clear that $\lambda\mapsto c(\lambda)$ is Lipschitz, thus it is differentiable almost everywhere. We will show a stronger claim that indeed the set of points where $c'(\lambda)$ does not exists is almost countable. Without loss of generality (from Theorem \ref{thm:limit}) We can assume $r(\lambda) = \lambda$ for $\lambda \in (-\varepsilon,\varepsilon)$ for some $\varepsilon>0$ in this section.

\begin{thm} Assume $\mathrm{(H1)},\mathrm{(H2)},\mathrm{(H3)}, \mathrm{(H4)},\mathrm{(A1)}$ and $\lambda\in (-\varepsilon,\varepsilon)$.
\begin{itemize}
    \item[(a)] The map $\lambda\mapsto c(\lambda)$ is left-differentiable and right-differentiable everywhere on its domain.
    \item[(b)] The left derivative $\lambda\mapsto c'_-(\lambda)$ is left continuous and the right derivative $\lambda\mapsto c'_+(\lambda)$ is right continuous on their domains.
    \item[(c)] The map $\lambda\mapsto c(\lambda)$ is differentiable except countably many points on its domain. 
\end{itemize}
\end{thm}
\begin{proof} For (a), by the same argument as in the proof of Theorem \ref{thm:limit} we see that $\lambda\mapsto c(\lambda)$ is left and right differentiable with 
\begin{equation*}
\begin{split}
    c_+'(\lambda) &= \max_{\mu\in \mathcal{M}_\lambda} \int_{\overline{\Omega}_\lambda\times \overline{B}_h} (-x)\cdot D_xL(x,v)d\mu(x,v)\\
    c_-'(\lambda) &= \min_{\mu\in \mathcal{M}_\lambda} \int_{\overline{\Omega}_\lambda\times \overline{B}_h} (-x)\cdot D_xL(x,v)d\mu(x,v)
\end{split}
\end{equation*}
where $\mathcal{M}_\lambda$ is the set of minimizing Mather measures on $\Omega_\lambda$.\\

\noindent
For $\lambda\in (-\varepsilon,\varepsilon)$ and $\nu_\lambda$ be any measure in $\mathcal{M}_\lambda$ then by the usual scaling as in Lemma \ref{lem:inM0} we have $-c(\lambda) = \left\langle \tilde{\nu}_\lambda,L((1+\lambda)x,v) \right\rangle$ and any subsequential weak limit $\tilde{\nu}_\lambda \rightharpoonup \nu_0$ in $\mathcal{P}(\overline{\Omega}\times \overline{B}_h)$ satisfies $\nu_0\in \mathcal{M}_0$. We claim further that 
\begin{equation*}
    \int_{\overline{\Omega}\times \overline{B}_h} (-x)\cdot D_xL(x,v)d\nu_0(x,v) = \begin{cases}
    c'_-(0) &\quad \text{if}\;\lambda\to 0^-,\\
    c'_+(0) &\quad \text{if}\;\lambda\to 0^+.
    \end{cases}
\end{equation*}
It is rather clear from \eqref{e:nice} and \eqref{e:nice2}, since, for instance if $\lambda\to 0^-$ then
\begin{align*}
    &\left\langle \mu, \frac{L((1+\lambda)x,v) - L(x,v)}{\lambda}\right\rangle + \frac{c(\lambda) - c(0)}{\lambda} \leq 0 \qquad\text{for all}\;\mu\in \mathcal{M}_0\\
    &\left\langle \tilde{\nu}_\lambda, \frac{L(x,v) - L((1+\lambda)x,v)}{\lambda}\right\rangle \leq \frac{c(\lambda) - c(0)}{\lambda}.
\end{align*}
Therefore, together with Theorem \ref{thm:limit} we deduce that
\begin{equation*}
    \begin{split}
       &c'_-(0) \leq \big\langle \mu, (-x)\cdot D_xL(x,v) \big\rangle \qquad\text{for all}\;\mu\in \mathcal{M}_0\\
       &  \big\langle \nu_0, (-x)\cdot D_xL(x,v)\big\rangle \leq c'_-(0).
    \end{split}
\end{equation*}
We conclude that 
\begin{equation*}
    \big\langle \nu_0, (-x)\cdot D_xL(x,v)\big\rangle = c'_-(0).
\end{equation*}
Now (b) follows easily. To see that $\lambda\mapsto c'_-(\lambda)$ is left continuous, it suffices to show it is left continuous at $0$. If $\lambda\to 0^-$, let $\nu_{\lambda}\in \mathcal{M}_\lambda$ that realizes $c'_-(\lambda)$, i.e.,
\begin{equation*}
\begin{split}
    c'_-(\lambda) &= \int_{\overline{\Omega}_\lambda\times \overline{B}_h} (-x)\cdot D_xL(x,v)d\nu_\lambda(x,v) = (1+\lambda) \int_{\overline{\Omega}\times \overline{B}_h} (-x)\cdot D_xL\big((1+\lambda)x,v\big)d\tilde{\nu}_\lambda(x,v).
\end{split}
\end{equation*}
From $\mathrm{(H4)}$ we have that $(-x)\cdot D_xL((1+\lambda)x,v)\to (-x)\cdot D_xL(x,v)$ uniformly on $\overline{\Omega}\times\overline{B}_h$, and since the limit of the right hand side is $c'_-(0)$ independent of subsequence, we deduce that
\begin{equation*}
    \lim_{\lambda \to 0^-} c_-'(\lambda) = c'_-(0).
\end{equation*}
The case $\lambda\to 0^+$ can be done in the same manner. Finally the fact that $\lambda\mapsto c(\lambda)$ is differentiable except countably many points is standard, since $\lambda \mapsto c'(\lambda)$ is defined almost everywhere and is non-decreasing, or one can argue as in \cite[Theorem 17.9]{hewitt_real_1965} or \cite[Theorem 4.2]{bruckner_differentiation_1978}. 
\end{proof}
With some additional information about the Hamiltonian, we can say something more about the map $\lambda\mapsto c(\lambda)$. 

\begin{lem}\label{thm:convexnew} Assume $\mathrm{(H1)}, \mathrm{(H2)}, \mathrm{(H3)}, \mathrm{(H4)}$, $\mathrm{(A1)}$ and further that $(x,p)\mapsto H(x,p)$ is \textit{jointly convex}, then $\lambda\mapsto c(\lambda)$ is convex. 
\end{lem}
We omit the proof of this lemma as it is a simple modification of Corollary \ref{cor:mycor}.

\section{The second normalization: convergence and a counter example}\label{sec5}
From Theorems \ref{thm:general} and \ref{thm:limit} we obtain the convergence of the second normalization \eqref{eq:familyclambda} when $\gamma$ is finite as in Corollary \ref{cor:second_norm}. In this section we provide an example where given any $r(\lambda)$, we can construct $\phi(\lambda)$ such that $\gamma = \infty$ and $\left\lbrace u_\lambda+\phi(\lambda)^{-1}c(\lambda)\right\rbrace_{\lambda>0}$ is divergent along some subsequence. To simplify notations, we will consider $r(\lambda) \geq 0$ and denote $c(\lambda)$ to be the eigenvalue of $H$ in $\Omega_\lambda = (1-r(\lambda)\Omega$. Let us consider the following Hamiltonian
\begin{equation}\label{eq:H_counter}
H(x,p) = |p| - V(x), \qquad (x,p)\in \overline{\Omega}\times \mathbb{R}^n,
\end{equation}
where $V:\overline{\Omega}\rightarrow \mathbb{R}$ is uniformly bounded continuous and is nonnegative. For a given $r(\lambda)$, we will construct $\phi(\lambda)$ so that $\left\lbrace u_\lambda+\phi(\lambda)^{-1}c(\lambda)\right\rbrace_{\lambda>0}$ is divergent as $\lambda \rightarrow 0^+$. The example is constructed based on an instability of the Aubry set $\mathcal{A}_{\Omega_\lambda}$ of $H$ on $\Omega_\lambda$, when $\lambda\rightarrow 0^+$. We recall from Theorem \ref{thm:V} that
\begin{equation*}
-c(0) = \min_{\overline{\Omega}} V \qquad\text{and}\qquad\mathcal{A}_{\Omega} = \left\lbrace x\in \overline{\Omega}: V(x) = \min_{\overline{\Omega} }V \right\rbrace. 
\end{equation*} 
Also, the Lagrangian is nonnegative in this case, since
\begin{equation}\label{L_counter}
    L(x,v) = \begin{cases}
    V(x)     &\qquad \text{if}\;|v|\leq 1,\\
    +\infty  &\qquad \text{if}\;|v|> 1.
    \end{cases}
\end{equation}
\begin{lem}\label{lem:equal} Assume \eqref{H3c}, \eqref{H4}, $\mathrm{(H3)}$ and $\mathrm{(A2)}$. Let
\begin{equation*}
    S_{\Omega}(x,y) = \sup \Big\lbrace u(x) - u(y): u\;\text{is a subsolution}\; H(x,Du(x))\leq c(0)\;\text{in}\;\Omega\Big\rbrace.
\end{equation*}
We can extend $S_{\Omega}$ uniquely to $\overline{\Omega}\times \overline{\Omega}$. If $\mathcal{A}_{\Omega} = \{z_0\}$ is a singleton then $u^0(x) \equiv S_\Omega(x,z_0)$ where $u^0$ is the maximal solution on $\Omega$ defined in Theorem \ref{thm:conv_bdd}.
\end{lem}
\begin{proof}[Proof of Lemma \ref{lem:equal}] One can show that $\mathcal{A}_{\Omega}$ is a uniqueness set for \eqref{eq:S_0} (see \cite{ fathi_pde_2005,ishii_vanishing_2020, jing_generalized_2020, mitake_uniqueness_2018}). From Lemma \ref{lem:sigma0} there exists $\sigma_0\in \mathcal{M}_0$ such that $\langle \sigma_0,u^0\rangle = 0$. If $\mathcal{A}_{\Omega} = \{z_0\}$ then we can show that $\mathrm{supp}\;(\sigma_0)\subset \{z_0\}$ and thus $\sigma_0 \equiv \delta_{z_0}$, hence 
\begin{equation*}
u^0(z_0) = \langle \sigma_0,u^0\rangle = 0 = S_\Omega(z_0,z_0).
\end{equation*}
Therefore $u^0(x)\equiv S_\Omega(x,z_0)$.
\end{proof}

\begin{defn}[Definition of the potential $V(x)$]\label{defV} We will construct a potential $V$ to use for the proof of Theorem \ref{thm:counter-example} on $\Omega = (-1,1)$. We start with the first step, the building block will be as follows.

\usetikzlibrary{arrows}
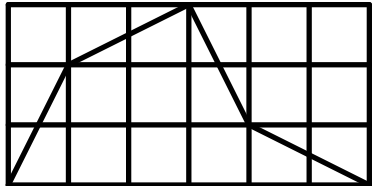
\begin{figure}[H]
   \centering
\begin{tikzpicture}[line cap=round,line join=round,>=triangle 45,x=0.4cm,y=0.4cm]
\clip(-10.,-0.) rectangle (10.,7.);
\draw [line width=2.pt] (-6.,6.)-- (6.,6.);
\draw [line width=2.pt] (6.,6.)-- (6.,0.);
\draw [line width=2.pt] (-6.,6.)-- (-6.,0.);
\draw [line width=2.pt] (-4.,6.)-- (-4.,0.);
\draw [line width=2.pt] (-2.,6.)-- (-2.,0.);
\draw [line width=2.pt] (2.,6.)-- (2.,0.);
\draw [line width=2.pt] (4.,6.)-- (4.,0.);
\draw [line width=2.pt] (-6.,4.)-- (6.,4.);
\draw [line width=2.pt] (6.,2.)-- (-6.,2.);
\draw [line width=2.pt] (0.,6.)-- (-4.,4.);
\draw [line width=2.pt] (-4.,4.)-- (-6.,0.);
\draw [line width=2.pt] (0.,6.)-- (2.,2.);
\draw [line width=2.pt] (2.,2.)-- (6.,0.);
\draw [line width=2.pt] (0.,6.)-- (0.,0.);
\draw [line width=2.pt] (-6.,0.)-- (6.,0.);
\end{tikzpicture}
\caption{The first step.}
   \label{fig1}
\end{figure}
Next, we apply the same construction but with a smaller scale, which gives us Figure \ref{fig2}.
\usetikzlibrary{arrows}
\begin{figure}[H]
   \centering
\begin{tikzpicture}[line cap=round,line join=round,>=triangle 45,x=0.4cm,y=0.4 cm]
\clip(-10.,-3.) rectangle (10.,7.);
\draw [line width=2.pt] (-6.,6.)-- (6.,6.);
\draw [line width=2.pt] (6.,6.)-- (6.,0.);
\draw [line width=2.pt] (-6.,6.)-- (-6.,0.);
\draw [line width=2.pt] (-4.,6.)-- (-4.,0.);
\draw [line width=2.pt] (-2.,6.)-- (-2.,0.);
\draw [line width=2.pt] (2.,6.)-- (2.,0.);
\draw [line width=2.pt] (4.,6.)-- (4.,0.);
\draw [line width=2.pt] (-6.,4.)-- (6.,4.);
\draw [line width=2.pt] (6.,2.)-- (-6.,2.);
\draw [line width=2.pt] (6.,0.)-- (6.,-3.);
\draw [line width=2.pt] (6.,0.)-- (9.,0.);
\draw [line width=2.pt] (9.,0.)-- (9.,-3.);
\draw [line width=2.pt] (9.,-3.)-- (6.,-3.);
\draw [line width=2.pt] (7.,0.)-- (7.,-3.);
\draw [line width=2.pt] (8.,0.)-- (8.,-3.);
\draw [line width=2.pt] (6.,-1.)-- (9.,-1.);
\draw [line width=2.pt] (9.,-2.)-- (6.,-2.);
\draw [line width=2.pt] (-9.,0.)-- (-9.,-3.);
\draw [line width=2.pt] (-9.,-3.)-- (-6.,-3.);
\draw [line width=2.pt] (-6.,-3.)-- (-6.,0.);
\draw [line width=2.pt] (-6.,0.)-- (-9.,0.);
\draw [line width=2.pt] (-8.,0.)-- (-8.,-3.);
\draw [line width=2.pt] (-7.,0.)-- (-7.,-3.);
\draw [line width=2.pt] (-9.,-1.)-- (-6.,-1.);
\draw [line width=2.pt] (-6.,-2.)-- (-9.,-2.);
\draw [line width=2.pt] (0.,6.)-- (-4.,4.);
\draw [line width=2.pt] (-4.,4.)-- (-6.,0.);
\draw [line width=2.pt] (0.,6.)-- (2.,2.);
\draw [line width=2.pt] (2.,2.)-- (6.,0.);
\draw [line width=2.pt] (-6.,0.)-- (-7.,-2.);
\draw [line width=2.pt] (-7.,-2.)-- (-9.,-3.);
\draw [line width=2.pt] (6.,0.)-- (8.,-1.);
\draw [line width=2.pt] (8.,-1.)-- (9.,-3.);
\draw [line width=2.pt] (0.,6.)-- (0.,0.);
\draw [line width=2.pt] (-6.,0.)-- (6.,0.);
\end{tikzpicture}
\caption{The second step.}
   \label{fig2}
\end{figure}
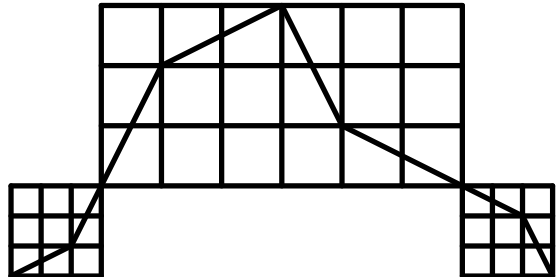

Keep switching the small box with this construction and with an appropriate initial length to start with, the graph of $V$ is given as in Figure \ref{fig3}.

\begin{figure}[H]
       \centering
      \includegraphics[scale=0.3]{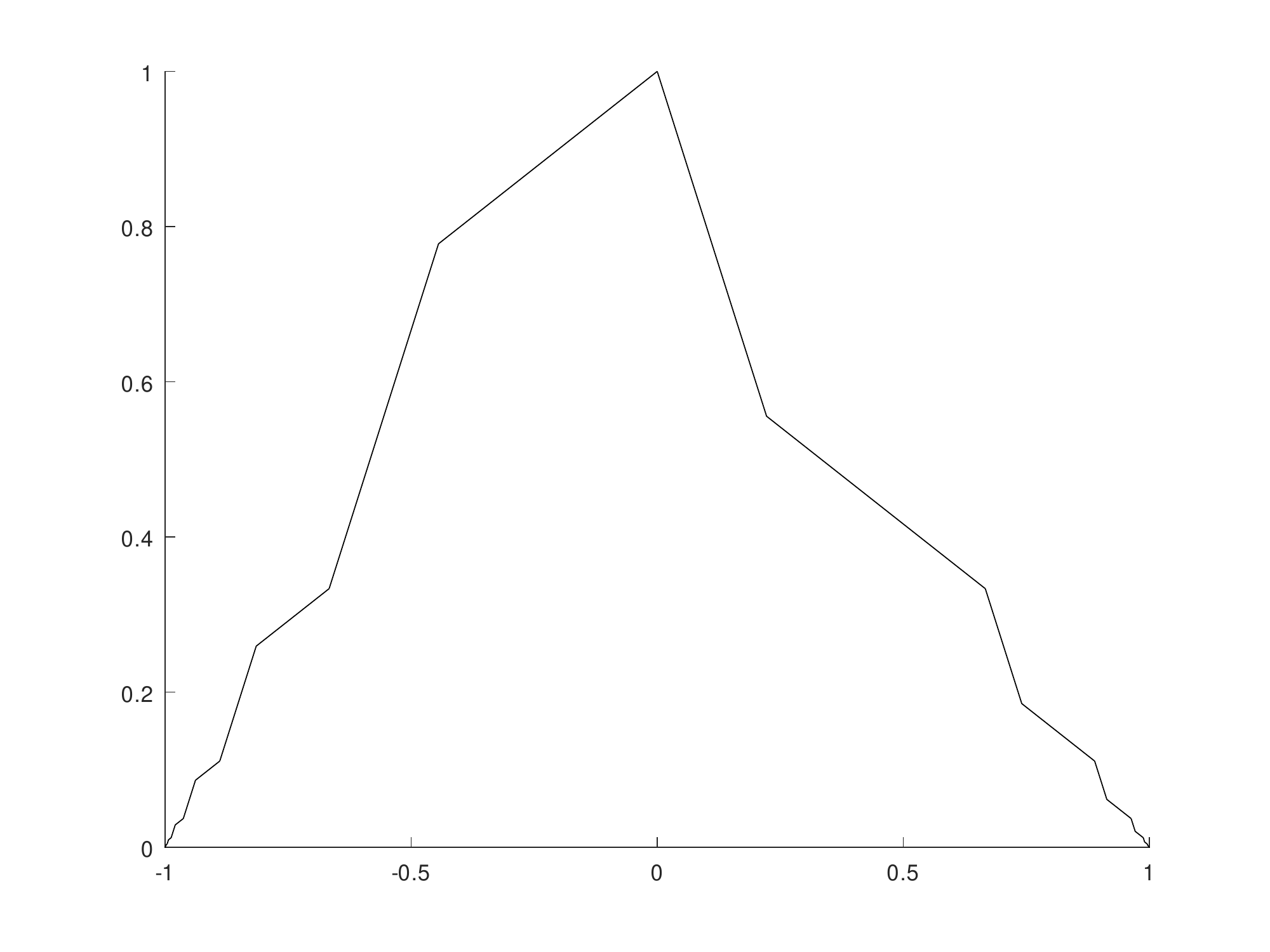}
       \caption{Graph of the function $V$.}
       \label{fig3}
\end{figure}
\end{defn}

\begin{lem}\label{lem:divergence} Let $V(x)$ defined as in Definition \ref{defV} and $\Omega_\lambda = (-1+r(\lambda), 1-r(\lambda))$. Then the maximal solution on $\Omega_\lambda$ (as in Theorem \ref{thm:conv_bdd}), denoted by $u^0_\lambda(x)$, does not converge as $\lambda\rightarrow 0^+$.
\end{lem}
\begin{proof}[Proof of Lemma \ref{lem:divergence}] By Theorem \ref{thm:V} the additive eigenvalue of $H$ on $\Omega_\lambda$, denoted by $c(\lambda)$, is given by $-c(\lambda) = \min_{x\in \overline{\Omega}_\lambda} V(x)$. By the construction of $V$, there are exactly two points, denoted by $z_\lambda^+$ and $z_\lambda^-$ such that 
\begin{equation}\label{eq:crucial0}
\left\lbrace z\in \overline{\Omega}_\lambda: V(z) =  \min_{x\in \overline{\Omega}_{\lambda}} V(x) = -c(\lambda) \right\rbrace = \big\lbrace z_\lambda^+,z_\lambda^-\big\rbrace.
\end{equation}
We can find two subsequence of $\lambda_j\rightarrow 0^+$ and $\delta_j\rightarrow 0^+$ such that $\lim_{\lambda_j\rightarrow 0^+} z_{\lambda_j} = -1$ and $\lim_{\delta_j\rightarrow 0^+} z_{\delta_j} = 1$. We claim that 
\begin{equation}\label{counter:claim}
    \lim_{\lambda_j\rightarrow 0^+} u_{\lambda_j}^0(x) = S_{\Omega}(x,-1) \qquad\text{and}\qquad  \lim_{\delta_j\rightarrow 0^+} u_{\delta_j}^0(x) = S_{\Omega}(x,1).
\end{equation}
For those $z_\lambda$ satisfying \eqref{eq:crucial0} we have $ u_{\lambda}^0(x) \equiv S_{\Omega_{\lambda}}\left(x, z_{\lambda}\right)$ for $x\in \overline{\Omega}_{\lambda}$. We show that \begin{equation}\label{eq:claim_counter}
    \lim_{\lambda_j\rightarrow 0} S_{\Omega_{\lambda_j}}\left(x,z_{\lambda_j}\right) = S_{\Omega}\left(x,z_0\right) \qquad\text{for}\;x\in \Omega
\end{equation}
where $z_0 = -1$. The other case is similar. If $x\in \Omega$ then for all $\lambda$ small enough we have $x\in \Omega_\lambda$, by Theorem \ref{lem:optimal} we have
\begin{align*}
    S_{\Omega_\lambda}(x,z_\lambda) &= \inf\left\lbrace \int_0^T \Big( c(\lambda)+ L(\xi(s),\dot{\xi}(s))\Big)ds: \xi\in \mathrm{AC}\left([0,T];\overline{\Omega}_\lambda\right), \xi(0) = z_\lambda, \xi(T) = x \right\rbrace,\\
    S_{\Omega}(x,z_\lambda) &= \inf\left\lbrace \int_0^T \Big(c(0)+L(\xi(s),\dot{\xi}(s))\Big)ds: \xi\in \mathrm{AC}\left([0,T];\overline{\Omega}\right), \xi(0) = z_\lambda, \xi(T) = x \right\rbrace.
\end{align*}
We show that $S_{\Omega_\lambda}(x,z_\lambda) \leq S_{\Omega}(x,z_0)$. Take any $\xi\in \mathcal{F}_{\Omega}(x,z_0; 0,T)$ (defined in Theorem \ref{lem:optimal}) and define $t_\lambda = \inf \; \big\lbrace s>0: \xi(s) = z_\lambda\big\rbrace  \in (0,T)$, then
\begin{equation*}
    \eta(s) = \begin{cases}
    z_\lambda &\qquad s\in [0,t_\lambda],\\
    \xi(s) &\qquad s\in [t_\lambda, T],
    \end{cases}
\end{equation*}
belongs to $\mathcal{F}_{\Omega_\lambda}(x,z_\lambda; 0,T)$, therefore together with \eqref{L_counter} we have
\begin{align*}
    \int_0^T \Big(c(0)+L(\xi(s),\dot{\xi}(s))\Big)ds &=  \int_0^{t_\lambda} \Big(c(0)+L(\xi(s),\dot{\xi}(s))\Big)ds + \int_0^T \Big(c(0)+L(\eta(s),\dot{\eta}(s))\Big)ds \\
    &\geq  \int_0^{t_\lambda} \Big(c(0)+L(\xi(s),\dot{\xi}(s))\Big)ds + \int_0^T \Big(c(\lambda)+L(\eta(s),\dot{\eta}(s))\Big)ds \\
    &\geq \max \Big\lbrace S_{\Omega}(z_\lambda,z_0),0\Big\rbrace +  S_{\Omega_\lambda}(x,z_\lambda).
\end{align*}
Therefore taking the infimum over all possible $\xi$ we deduce that
\begin{equation*}
    S_{\Omega}(x,z_0) \geq \max \Big\lbrace S_{\Omega}(z_\lambda,z_0),0\Big\rbrace +  S_{\Omega_\lambda}(x,z_\lambda)
\end{equation*}
and thus
\begin{equation}\label{eq:couter_1}
    \limsup_{\lambda\rightarrow 0^+} S_{\Omega_\lambda}(x,z_\lambda) \leq S_{\Omega}(x,z_0).
\end{equation}
Now let us start with $\xi_n\in \mathcal{F}_{\Omega_\lambda}(x,z_\lambda; 0, T_n)$ such that
\begin{equation}\label{eq:couter_2}
    \int_0^{T_n} \Big(L(\xi(s),\dot{\xi}(s))+c(\lambda)\Big)ds < S_{\Omega_\lambda}(x,z_\lambda) + \frac{1}{n}.
\end{equation}
Let us connect $z_0$ and $z_\lambda$ by the straight line $\zeta(s) = (1-s)z_0 + sz_\lambda$ for $s\in [0,1]$. Since $|\dot{\zeta}(s)| = |z_0-z_\lambda| \ll 1$, therefore from \eqref{L_counter} we have
\begin{equation*}
    \int_0^1 L\left(\zeta(s),\dot{\zeta}(s)\right)\;ds = \int_0^1 V(\zeta(s))\;ds \leq \max_{x\in [z_0,z_\lambda]} V(x) = -c(\lambda).
\end{equation*}
Therefore
\begin{equation}\label{eq:couter_3}
    \int_0^1 \Big(L(\zeta(s),\dot{\zeta}(s)) + c(\lambda)\Big)ds \leq 0.
\end{equation}
Let us define
\begin{equation*}
    \eta_n(s) = 
    \begin{cases}
    \zeta(s) &\qquad\text{for}\; s\in [0,1],\\
    \xi(s-1) &\qquad\text{for}\; s\in [1,T_n+1]
    \end{cases}
\end{equation*}
then $\eta_n\in \mathcal{F}_{\Omega}(x,z_0;0,T_{n+1})$. From \eqref{eq:couter_2} and \eqref{eq:couter_3} we have
\begin{align*}
S_{\Omega_\lambda}(x,z_\lambda) + \frac{1}{n} &>  \int_0^1 \Big(L(\zeta(s),\dot{\zeta}(s)) + c(\lambda)\Big)ds  + \int_0^{T_n} \Big(L(\xi(s),\dot{\xi}(s))+c(\lambda)\Big)ds\\
&= \int_0^{T_n+1} \Big(L(\eta_n(s),\dot{\eta}_n(s)) + c(\lambda)\Big)ds\geq S_{\Omega}(x,z_0)
\end{align*}
since $\eta_n\in \mathcal{F}_{\Omega}(x,z_0;0,T_{n+1})$. Let $\lambda\rightarrow 0^+$ and then $n\rightarrow \infty$ we have
\begin{equation}\label{eq:couter_4}
    \liminf_{\lambda\rightarrow 0^+} S_{\Omega_{\lambda}}(x,z_\lambda) \geq S_{\Omega}(x,z_0).
\end{equation}
From \eqref{eq:couter_1} and \eqref{eq:couter_4} we obtain \eqref{eq:claim_counter} and \eqref{counter:claim} follows. We finally observe that $S_\Omega(x,-1) \neq S_\Omega(x,1)$, since otherwise $S_\Omega(-1,1) = S_\Omega(1,-1) = 0$, which is impossible. Indeed, if $\xi\in \mathcal{F}_\Omega(1,-1;0,T)$ then, as \eqref{L_counter} implies that $|\dot{\xi}(s)|\leq 1$ a.e., we deduce from \eqref{L_counter} that 
\begin{equation*}
    \int_0^T L(\xi(s),\dot{\xi}(s))ds = \int_0^T V(\xi(s))ds \geq \int_0^T V(\xi(s))\dot{\xi}(s)ds = \int_{-1}^1 V(x)\;dx = \Vert V\Vert_{L^1(\Omega)} > 0.
\end{equation*}
Therefore $S_\Omega(-1,1) > 0$.
\end{proof}

 \begin{proof}[Proof of Theorem \ref{thm:counter-example}] Let $H$ be defined as in \eqref{eq:H_counter}, we consider the following discounted problems:
\begin{equation}\label{eq:SP1}
    \begin{cases}
    \delta u_\delta(x) + H(x,Du_\delta(x)) \leq 0 \qquad\text{in}\;\Omega_\lambda,\\
    \delta u_\delta(x) + H(x,Du_\delta(x)) \geq 0 \qquad\text{on}\;\overline{\Omega}_\lambda.
    \end{cases}
\end{equation}
Let $c(\lambda)$ be the eigenvalue of $H$ over $\Omega_\lambda$. By Theorem \ref{thm:conv_bdd} we know that
\begin{equation*}
   \lim_{\delta\rightarrow 0^+}\left( u_\delta(x) + \frac{c(\lambda)}{\delta}\right) \rightarrow u^0_\lambda(x) 
\end{equation*}
uniformly on $\overline{\Omega}$, where $u^0_\lambda(x)$ is a maximal solution on $\Omega_\lambda$. For each $\lambda>0$, we can find $\tau(\lambda)>0$ such that
\begin{equation}\label{eq:finalcrucial}
    \sup_{x\in \overline{\Omega}_{\lambda}} \left|\left(u_\delta(x) + \frac{c(\lambda)}{\delta}\right) - u^0_\lambda(x) \right|\leq r(\lambda)\qquad\text{for all}\; \delta \leq \tau(\lambda).
\end{equation}
Set $\phi(\lambda) = \tau(\lambda)r(\lambda)^2$, then $\phi(\lambda)\rightarrow 0$ as $\lambda\rightarrow 0^+$ and $\gamma = \infty$. The function $\phi(\lambda)$ can be modified to be decreasing. Now by \eqref{eq:finalcrucial} and Lemma \ref{lem:divergence}, along two subsequences $\lambda_j$ and $\delta_j$ we have
\begin{equation*}
    \lim_{\lambda_j\rightarrow 0^+} \left(u_{\lambda_j}(x) + \frac{c(\lambda_j)}{\phi(\lambda_j)}\right) = S_{\Omega}(x,-1) \neq S_\Omega(x,1) = \lim_{\delta_j\rightarrow 0^+} \left(u_{\delta_j}(x) + \frac{c(\delta_j)}{\phi(\delta_j)}\right).
\end{equation*}
Thus we have the divergence of $\left\lbrace u_\lambda+\phi(\lambda)^{-1}c(\lambda) \right\rbrace _{\lambda > 0}$ in this case.
\end{proof}
 
\begin{rem} Note that the parametrization here $u_\lambda$ means $u_{\phi(\lambda)}$, which is the same as in the original definition \eqref{eq:def_ulambda}. In \eqref{eq:def_ulambda} we should have used $u_{\phi(\lambda)}$ instead of $u_\lambda$ but we simplify the notation for clarity.
\end{rem}

\section*{Appendix}

\begin{lem}\label{thm:Ishii} Assume that $\Omega$ is bounded, open and $0\in \Omega$. Assume further that \eqref{condA2} holds for some $\kappa >0$, then $\Omega$ is star-shaped and $\mathrm{(A2)}$ holds.
\end{lem}
\begin{proof}[Proof of Lemma \ref{thm:Ishii}] Suppose that $\Omega$ is not star-shaped, there exists $x\in \overline{\Omega}$ and $0< \theta < 1$ such that $\theta x\notin \Omega$. Since $0$ is an interior point of $\Omega$, there exists $0<\delta<\theta$ such that $\tau x \in \Omega$ for all $0<\sigma \leq \delta$. Let us define $\eta = \sup \big\lbrace \tau>0: \tau x \in \Omega \big\rbrace$ then $0< \delta \leq \eta \leq \theta$ and $\eta x\in \partial \Omega$. Set $y = \eta x\in \partial \Omega$, we see that 
\begin{equation*}
    x = \eta^{-1}y = (1+r)y \in (1+r)\partial\Omega
\end{equation*}
where $\eta^{-1} = 1 + r$. Now \eqref{condA2} gives us that $0=\mathrm{dist}(x,\Omega) \geq \kappa r$ which is a contradiction and thus $\Omega$ is star-shaped. 

For $0<r<1$, as $\Omega$ is star-shaped, $(1-r)\overline{\Omega}\subset(1+r)^{-1}\Omega$ and $B\left(0,\frac{\kappa r}{1+r}\right)\subset B\left(0,\frac{\kappa r}{2}\right)$. From \eqref{condA2} we have $\Omega+B(0,\kappa r) \cap (1+r)\partial \Omega =\emptyset$ for all $r\in (0,1)$, therefore 
\begin{equation*}
     (1+r)^{-1}\Omega + B\left(0,\frac{\kappa r}{1+r}\right) \cap \partial\Omega = \emptyset \quad\Longrightarrow\quad (1-r)\overline{\Omega}+ B\left(0,\frac{\kappa r}{2}\right) \cap \partial \Omega = \emptyset.
\end{equation*}
From \eqref{condA2} we deduce that $(1-r)\overline{\Omega}+ B\left(0,\frac{\kappa r}{2}\right) \subset \Omega$. We observe that
\begin{equation*}
    B\left(x - rx, \frac{\kappa r}{2}\right) = (1-r)x+B\left(0,\frac{\kappa r}{2}\right) \subset (1-r)\overline{\Omega}+ B\left(0,\frac{\kappa r}{2}\right) \subset \Omega.
\end{equation*}
This implies $\mathrm{(A2)}$ with $\eta(x) = -x$, $r = \frac{\kappa}{2}$ and $h = \frac{1}{2}$.
\end{proof}

\begin{proof}[Proof of Theorem \ref{thm:pre}] By the priori estimate $\delta |u_\delta(x)| + |Du_\delta(x)|\leq C_H$ for $x\in \overline{\Omega}$. Fix $x_0\in \overline{\Omega}$, then by Aezel\`a-Ascoli theorem there exists a subequence $\delta_j$ and $u\in \mathrm{C}(\overline{\Omega})$ such that $u_{\delta_j}(\cdot) - u_{\delta_j}(x_0) \rightarrow u(\cdot)$ uniformly on $\overline{\Omega}$ for some $u\in \mathrm{C}(\overline{\Omega})$. By Bolzano-Weiertrass theorem there exists $c\in \mathbb{R}$ such that (upto subsequence) $\delta_ju_{\delta_j}(x_0)\rightarrow -c$. 

By stability of viscosity solution we have $H(x,Du(x)) = c$ in $\Omega$. We will show $H(x,Du(x))\geq c$ on $\overline{\Omega}$. Let $\tilde{x}\in \partial\Omega$ and $\varphi\in \mathrm{C}^1(\overline{\Omega})$ such that $u-\varphi$ has a strict minimum over $\overline{\Omega}$ at $\tilde{x}$, we show that $H(\tilde{x},D\varphi(\tilde{x})) \geq c$. Without loss of generality we can assume that $(u-\varphi)(x) \geq (u-\varphi)(\tilde{x}) = 0$ for $x\in \overline{\Omega}$.

Define $\varphi_\delta(x) = (1+\delta)\varphi\left(\frac{x}{1+\delta}\right)$ for $x\in (1+\delta)\overline{\Omega}$. Let us define 
\begin{equation*}
\Phi(x,y) = \varphi_\delta(x) - u_\delta(y) - \frac{|x-y|^2}{2\delta^2}, \qquad (x,y) \in (1+\delta)\overline{\Omega}\times \overline{\Omega}.
\end{equation*}
Assume $\Phi$ has maximum over $(1+\delta)\overline{\Omega}\times \overline{\Omega}$ at $(x_\delta,y_\delta)$. As $\Phi(x_\delta,y_\delta)\geq \Phi(y_\delta,y_\delta)$, we obtain $|x_\delta - y_\delta|\leq C\delta$. By compactness we deduce that $(x_\delta,y_\delta) \rightarrow (\overline{x},\overline{x})$ for $\overline{x}\in \overline{\Omega}$ as $\delta \rightarrow 0^+$. We deduce further that
\begin{equation*}
\limsup_{\delta\rightarrow 0}\frac{|x_\delta - y_\delta|^2}{2\delta^2} \leq \limsup_{\delta \rightarrow 0} \Big(\varphi(x_\delta) - \varphi(y_\delta)\Big) = 0 \quad\Longrightarrow\quad |x_\delta - y_\delta| = o(\delta).
\end{equation*}
Also $\Phi(x_\delta, y_\delta) \geq \Phi(\tilde{x},\tilde{x})$, let $\delta \rightarrow 0$ we have $\Phi(\overline{x},\overline{x}) \geq \Phi(\tilde{x},\tilde{x})$ which implies that $\overline{x} = \tilde{x}$. By $\mathrm{(A2)}$ we deduce that $x_\delta \in (1+\delta)\Omega$. Now by supersolution test as $y\mapsto \Phi^\delta(x_\delta,y)$ has a max at $y_\delta$, we obtain
\begin{equation*}
\delta u_\delta(y_\delta) + H\Big(y_\delta, \delta^{-2}(x_\delta - y_\delta)\Big) \geq 0.
\end{equation*}
As $x\mapsto \Phi(x,y_\delta)$ has a max at $x_\delta \in (1+\delta)\Omega$ as an interior point of $(1+\delta)\Omega$, we deduce that $D\varphi_\delta(x_\delta) = \delta^{-2}(x_\delta - y_\delta)$. Therefore
\begin{equation*}
\delta u_\delta(y_\delta) + H\Big(y_\delta, D\varphi_\delta(x_\delta)\Big) \geq 0.
\end{equation*}
As $u_\delta(\cdot)$ is Lipschitz with constant $C_H$, we have $u_\delta(y_\delta)\rightarrow u(\tilde{x})$ along the subsequence $\delta_j$. Therefore as $\delta_j\rightarrow 0$ we have $H(\tilde{x}, D\varphi(\tilde{x})) \geq c$.

Now with the help of comparison principle, we obtain the uniqueness of $c = c(0)$ and thus the convergence of the full sequence $\delta u_\delta(x_0)\rightarrow -c(0)$ follows. If we use the following normalization \begin{equation*}
    \lim_{j\rightarrow \infty} \left(u_{\delta_j}(x) + \frac{c(0)}{\delta_j}\right) = w(x)
\end{equation*}
then by a similar argument we can show $w$ solves \eqref{eq:S_0} as well, and 
\begin{align*}
   u(x) &= \lim_{j\rightarrow \infty}\left( u_{\delta_j}(x) - u_{\delta_j}(x_0)\right) \\
   &= \lim_{j\rightarrow \infty}\left( u_{\delta_j}(x) + \frac{c(0)}{\delta_j}\right) - \lim_{j\rightarrow \infty}\left( u_{\delta_j}(x_0) + \frac{c(0)}{\delta_j}\right) = w(x) - w(x_0).
\end{align*}
We have left to show \eqref{eq:rate0}. Let $u$ be defined as the limit of $u_{\delta_j}(\cdot) - u_{\delta_j}(x_0)$, we have $|u(x)|+|Du(x)|\leq C$ for $x\in \Omega$ where $C$ depends on $C_H$ and $\mathrm{diam}(\Omega)$. It is clear that $u(x)-\delta^{-1}c(0)\pm C$ are, respectively, subsolution and supersolution to \eqref{state-def}, therefore by comparison principle we obtain \eqref{eq:rate0}.
\end{proof}

\begin{proof}[Proof of Theorem \ref{thm:V}] If $v\in \mathrm{C}(\overline{\Omega})$ is a solution to \eqref{eq:E} then for a.e. $x\in \overline{\Omega}$ we have $-V(x)\leq|Dv(x)|-V(x) = c_\Omega$, therefore $c_\Omega \geq \max_{\overline{\Omega}}(-V) = -\min_{\overline{\Omega}} V$. Assume $V$ attains its minimum over $\overline{\Omega}$ at $x_0$ then by supersolution test at that point we have $0\geq -V(x_0)\geq c_\Omega$, therefore $c_\Omega = -\min_{\overline{\Omega}} V$.

Let $z\in \overline{\Omega}$ such that $V(z) = -c_\Omega$, we check that $x\mapsto S_\Omega(x,z)$ is a supersolution at $x=z$. Let $\omega(\cdot)$ be the modulus of continuity of $V$ on $\overline{\Omega}$, we have $|V(x)+c_\Omega|\leq \omega(r)$ for all $x\in B(z,r)\cap \overline{\Omega}$. From \eqref{eq:E} as $x\mapsto u(x) = S_\Omega(x,z)$ is a subsolution in $\Omega$, we have
\begin{equation*}
    |Du(x)|-V(x) \leq c_\Omega\qquad\Longrightarrow\qquad |Du(x)|\leq V(x)+c_\Omega\leq \omega(r) 
\end{equation*}
for a.e. $x\in B(z,r)\cap \overline{\Omega}$ and for all $r>0$, thus
\begin{equation*}
    |u(x)| = |u(x)-u(z)|\leq \int_0^1 |Du(sx + (1-s)z)\cdot (x-z)|\;ds \leq \omega(r)r
\end{equation*}
for $x\in B(z,r)\cap \overline{U}$. That means $x\mapsto u(x)$ is differentiable at $x=z$ and $Du(z) = 0$, thus $x\mapsto u(x) = S_\Omega(x,z)$ is a solution to \eqref{eq:E}.

Conversely, if $V(z)=-c_\Omega + \varepsilon$ for some $\varepsilon>0$, then at $x=z$ we have $0\in D^-u(z)$ where $u(x) = S_\Omega(x,z)$, therefore if the supersolution test holds then we must have $-V(z) \geq c_\Omega$, hence $\varepsilon < 0$ which is a contradiction, thus $x\mapsto S_\Omega(x,z)$ fails to be a supersolution at $x=z$. 
\end{proof}

\begin{proof}[Proof of Theorem \ref{thm:eigenvalue}] Without loss of generality we assume $c_U = 0$. Let $z\in \mathcal{A}_U \subset \Omega$ and $w(x) = S_U(x,z)$ solves \eqref{eq:E}, we have $H(x,Dw(x)) = 0$ in $\Omega$. We have
\begin{equation*}
c_\Omega = \inf \big\lbrace c\in \mathbb{R}: H(x,Du(x)) = c\;\text{admits a viscosity subsolution in}\;\Omega\big\rbrace \leq 0.
\end{equation*}
Assume the contrary that $c_\Omega < 0$ then there exists $u\in \mathrm{C}(\overline{\Omega})\cap \mathrm{W}^{1,\infty}(\Omega)$ solves $H(x,Du(x)) \leq c(0) < 0$ in $\Omega$. Let us consider $g(x) = w(x)$ defined for $x\in \partial \Omega$ and the boundary value problem
\begin{equation}\label{eq:bdrS0}
\begin{cases}
H(x, Dv(x)) = 0 &\quad\;\text{in}\;\Omega,\\ 
\;\,\qquad\qquad v = g &\quad\;\text{on}\;\partial \Omega.
\end{cases}
\end{equation}
As there exists a solution $u$ such that $H(x,Du(x)) < 0$ in $\Omega$, by Theorem \ref{thm:CP_Dirichlet} the problem \eqref{eq:bdrS0} cannot have more than one solution. On the other hand, the following function is a solution to \eqref{eq:bdrS0}
\begin{equation*}
\mathcal{V}(x) = \min_{y\in \partial \Omega} \Big\lbrace g(y) + S_U(x,y) \Big\rbrace.
\end{equation*}
Indeed, for each $y\in \partial \Omega$, $x\mapsto g(y)+S_U(x,y)$ is a Lipschitz viscosity solution to $H(x,Dv(x)) = 0$ in $\Omega$, therefore by the convexity of $H$ is obtain $\mathcal{V}$ is a viscosity solution to $H(x,D\mathcal{V}(x)) = 0$ in $\Omega$ as well. On the boundary we see that $\mathcal{V}(x) \leq g(x)$, and also for any $y\in \partial\Omega$ then 
\begin{equation*}
    g(y)+S_U(x,y) = S_U(y,z) + S_U(x,y) \geq S_U(x,z) = g(x),
\end{equation*}
which implies that $\mathcal{V}(x) = g(x)$ on $\partial\Omega$. Therefore we must have $\mathcal{V}(x) = S_U(x,z)$ for all $x\in \overline{\Omega}$, hence $\mathcal{V}(z) = S_U(z,z) = 0$ and as a consequence there exists $y\in \partial \Omega$ such that
\begin{equation*}
S_U(y,z) + S_U(z,y) = 0.
\end{equation*}
This implies that $y\in \mathcal{A}_U$ (see \cite{fathi_pde_2005, ishii_vanishing_2020}), which is a contradiction since $\mathcal{A}_U$ is supported inside $\Omega$, therefore we must have $c_\Omega = 0$. 
\end{proof}

\begin{proof}[Proof of Lemma \ref{lem:regu}] From $\mathrm{(H4)}$ for each $R>0$ there is a nondecreasing function $\omega_R:[0,\infty)\rightarrow [0,\infty)$ with $\omega_R(0) = 0$ such that $ |D_xL(x,v) - D_xL(y,v)|\leq \omega_R(|x-y|)$ if $|x|,|v|\leq R$. Fix $(x,v)\in \overline{\Omega}\times \overline{B}_h$, we can assume $h$ is large so that $\overline{\Omega}\subset \overline{B}_h$. Let $f(\delta) = L\left((1-\delta)x,v\right)$ then $\delta\mapsto f(\delta)$ is continuously differentiable and 
\begin{align*}
    \left|\frac{f(\delta) - f(0)}{\delta} - f'(0)\right| &\leq \sup_{s\in [0,\delta]}|f'(s) - f'(0)|\\
    &= \sup_{s\in [0,\delta]} |x|.\left|D_xL\left((1-s)x,v\right) - D_xL(x,v)\right| \leq \big(\mathrm{diam}\;\Omega\big) \omega_h\left(\delta|x|\right).
\end{align*}
Therefore
\begin{equation*}
    \lim_{\delta\rightarrow 0^+} \left( \sup_{(x,v)\in \overline{\Omega}\times \overline{B}_h} \left|\frac{L\left((1-\delta)x,v\right) - L(x,v)}{\delta} - (-x)\cdot D_xL(x,v)\right|\right) = 0.
\end{equation*}
The conclusions follow from here.
\end{proof}
\begin{thm}[Comparison principle for Dirichlet problem, \cite{Bardi1997}]\label{thm:CP_Dirichlet} Let $\Omega$ be a bounded open subset of $\mathbb{R}^n$. Assume $u_1,u_2\in \mathrm{C}(\overline{\Omega})$ are, respectively viscosity subsolution and supersolution of $H(x,Du(x)) = 0$ in $\Omega$ with $u_1\leq u_2$ on $\partial \Omega$. Assume further that 
\begin{itemize}
    \item $|H(x,p) - H(y,p)|\leq \omega((1+|p|)|x-y|)$ for all $x,y\in \Omega$ and $p\in \mathbb{R}^n$.
    \item $p\mapsto H(x,p)$ is convex for each $x\in \Omega$.
    \item There exists $\varphi\in \mathrm{C}(\overline{\Omega})$ such that $\varphi\leq u_2$ in $\overline{\Omega}$ and $H(x,D\varphi(x)) < 0$ in $\Omega$.
    \end{itemize}
    Then $u_1\leq u_2$ in $\Omega$.
\end{thm}

\section*{Acknowledgments}
The author would like to express his appreciation to his advisor, Hung V. Tran for giving him this interesting problem and for his invaluable guidance and patience. The author also would live to thanks Hiroyoshi Mitake for useful discussion on the subject when he visited Madison in September 2019, and for many great comments including suggestions that lead the oscillating behavior of $r(\lambda)$ and Corollary \ref{cor:equala}. The author also would like to thanks Hitoshi Ishii for many useful suggestions, including Lemma \ref{thm:Ishii} and a different perspective on Theorem \ref{thm:general}. Finally, the author would like to thank Michel Alexis, Dohyun Kwon for useful discussion and Jingrui Cheng for pointing our mistakes in the earlier version.\\


\bibliography{zzzzlibrary.bib}{}
\bibliographystyle{acm}


\end{document}